\documentclass[a4paper]{article}
\usepackage[utf8]{inputenc}

\usepackage[top=2.5cm, bottom=2.5cm, left=2.5cm, right=2.5cm]{geometry}

\usepackage[dvipsnames]{xcolor}

\usepackage{amsthm, amsmath, amssymb, mathrsfs, hyperref, tikz}
\hypersetup{citecolor=purple, linkcolor=blue, colorlinks=true}
\usepackage[noabbrev]{cleveref}
\usepackage[shortlabels]{enumitem}
\newtheorem{thm}{Theorem}[section]
\newtheorem{lm}[thm]{Lemma}

\newtheorem{crl}[thm]{Corollary}
\newtheorem{prop}[thm]{Proposition}

\theoremstyle{definition}

\newtheorem{rmk}[thm]{Remark}
\newtheorem{df}[thm]{Definition}
\newtheorem{prob}[thm]{Problem}

\newcommand{\eps}{\varepsilon}

\newcommand{\FF}{\mathbb F}

\newcommand{\set}[1]{ \left \{ #1 \right \} }
\newcommand{\sett}[2]{ \left\{ #1 \, \, || \, \, #2 \right \} }

\newcommand{\pg}{\textnormal{PG}}
\newcommand{\ag}{\textnormal{AG}}

\newcommand{\floor}[1]{\left \lfloor #1 \right \rfloor}
\newcommand{\ceil}[1]{\left \lceil #1 \right \rceil}

\renewcommand{\mp}{\mathcal P}
\newcommand{\mb}{\mathcal B}

\newcommand{\one}{\mathbf 1}
\newcommand{\zero}{\mathbf 0}

\newcommand{\al}{\alpha}
\newcommand{\lam}{\lambda}
\newcommand{\E}{\mathbb{E}}
\newcommand{\sub}{\subseteq}
\newcommand{\sm}{\setminus}
\newcommand{\half}{\frac{1}{2}}

\DeclareMathOperator{\sbp}{sbp}

\usepackage[multiple]{footmisc}

\title{Incidence-free sets and edge domination in incidence graphs}

\author{
	\begin{tabular}{l}
		Sam \textbf{S}piro\footnote{The ordering of the authors is done for purely aesthetic reasons and is independent of the significance of each of their contributions to this paper.}~\footnote{\small Department~of Mathematics, Rutgers University, Piscataway, NJ 08854. sas703@scarletmail.rutgers.edu. This material is based upon work supported by the National Science Foundation Mathematical Sciences Postdoctoral Research Fellowship under Grant No.\ DMS-2202730}\\
		Sam \textbf{A}driaensen\thanks{\small Department~of Mathematics and Data Science,
			Vrije Universiteit Brussel, Pleinlaan 2 1050 Brussels, Belgium. Sam.Adriaensen@vub.be}\\ 
		Sam \textbf{M}attheus\thanks{\small Department~of Mathematics, University of California, San Diego,	La Jolla, CA, 92093-0112. smattheus@ucsd.edu}
	\end{tabular}
}

\date{\today}

\begin{document}
	
	\maketitle
	
	\begin{abstract}
		A set of edges $\Gamma$ of a graph $G$ is an edge dominating set if every edge of $G$ intersects at least one edge of $\Gamma$, and the edge domination number $\gamma_e(G)$ is the smallest size of an edge dominating set. Expanding on work of Laskar and Wallis, we study $\gamma_e(G)$ for graphs $G$ which are the incidence graph of some incidence structure $D$, with an emphasis on the case when $D$ is a symmetric design. In particular, we show in this latter case that determining $\gamma_e(G)$ is equivalent to determining the largest size of certain incidence-free sets of $D$. Throughout, we employ a variety of combinatorial, probabilistic
		and geometric techniques, supplemented with tools from spectral graph theory. 
	\end{abstract}
	
	\noindent {\bf Keywords:} edge domination, design, incidence structure, matching, incidence-free sets. \\ 
	{\bf MSC2020 Classification:} 05B05, 05C70.

	%{\let\thefootnote\relax\footnotetext{The ordering of the authors is for purely aesthetic reasons and is independent of the significance of each of their contributions to this paper.}}
	
	\section{Introduction}
	
	Roughly speaking, the area of \textit{extremal combinatorics} centres around problems which ask how large or small a given combinatorial object can be under certain restrictions.  For example, Mantel's theorem states that the maximum number of edges of an $n$-vertex triangle-free graph is $\floor{n^2/4}$.  In this paper,  the extremal questions we consider are domination-type problems, which roughly involves studying the minimum number of vertices or edges one needs to ``cover'' a graph $G$.  Below we recall the basic definitions for domination, and we refer the interested reader to the book by Haynes, Hedetniemi, and Slater~\cite{haynes2013fundamentals} for more on this subject. \\

	A \emph{dominating set} in a graph $G = (V,E)$ is a set of vertices $S \subseteq V$ such that each vertex is either contained in $S$, or has a neighbour in $S$.
	The size of the smallest dominating set is called the \emph{domination number} of $G$ and is denoted as $\gamma(G)$.  A large body of work is dedicated to studying $\gamma(G)$ as well as its many variants such as roman domination \cite{cockayne2004roman} and total domination \cite{henning2013total}.
	
	A related concept is an \emph{edge dominating set} in a graph $G=(V,E)$.
	This is a subset $\Gamma \subseteq E$ of edges, such that for each edge $e \in E$, there exists an edge $e' \in \Gamma$ with $e' \cap e \neq \emptyset$.
	We note that this could be equivalently defined as a dominating set in the line graph of $G$.
	
	The \emph{edge domination number} of $G$ is defined as the size of the smallest edge dominating set, and will be denoted by $\gamma_e(G)$.  Again a large body of work is dedicated to studying $\gamma_e(G)$, see for example \cite{horton1993minimum,gavril}.\\  
	
	In this paper we study domination type problems for graphs which come from incidence structures.  An \emph{incidence structure} is a pair $D = (\mp,\mb)$, such that $\mb$ is a collection of subsets of $\mp$.
	We say that $P \in \mp$ and $B \in \mb$ are \emph{incident} if $P \in B$.
	The elements of $\mp$ and $\mb$ are called \emph{points} and \emph{blocks} respectively.
	We say that $D$ is of \emph{type} $(v,b,k,r,\lambda)$ if
	\begin{itemize}
		\item $|\mp| = v$, and $|\mb| = b$,
		\item every block is incident with $k$ points, every point is incident with $r$ blocks,
		\item $\lambda \ge 1$ is the maximum number such that there exist two distinct points $P$ and $Q$ which have $\lambda$ blocks incident with both $P$ and $Q$.
	\end{itemize}
	The convention $\lambda\ge 1$ is somewhat non-standard and is adopted to exclude trivial incidence structures. From these conditions, we have $v r = b k$ and $(v-1)\lam \geq r(k-1)$.
	If $v=b$ (which implies $r=k$), and if $\lambda$ is also the maximum number of points incident with two distinct blocks, then we call $D$ a \emph{symmetric incidence structure}, abbreviated SIS, of type $(v,k,\lambda)$.
	
	If any two distinct points of an incidence structure are incident with exactly $\lambda$ blocks, then we call $D$ a $(v,k,\lambda)$-\emph{design}, and in this case
	\begin{align}
		\label{EqDesigns}
		v r = b k, && (v-1)\lam = r(k-1).
	\end{align}
	Fisher's inequality states that in this case we always have $v \leq b$.
	In case of equality, i.e.\ if $D$ is a design with $v=b$, then any two blocks of $D$ intersect in exactly $\lambda$ points, and $D$ is called a \emph{symmetric design}. We refer the reader to \cite{colbourndinitz} for proofs of these statements and other background on designs.
	
	The \emph{incidence graph} $I(D)$, also called the \emph{Levi graph}, of an incidence structure $D=(\mp,\mb)$ is the bipartite graph with bipartition $\mp$ and $\mb$, where a point $P$ and a block $B$ are adjacent if and only if they are incident.   \\ 
	
	It turns out that the problem of determining the domination number of $I(D)$ is closely related to studying the smallest blocking sets and covers in $D$.  Partially because of this, there have been several papers in recent years investigating the domination number of incidence graphs of block designs and other incidence structures.
	See, for example, \cite{goldberg, tang} for incidence graphs of designs, \cite{hegernagy} for incidence graphs of projective planes, and \cite{hegerhernandezlucas} for incidence graphs of generalised quadrangles.
	
	Some work for edge domination numbers of incidence graphs of projective planes was done by Laskar and Wallis \cite{laskar}, but outside of this very little is known about edge domination numbers of incidence graphs.  Our goal for this paper is to expand upon this literature by focusing on the following problem:
	\begin{center}
		\textit{Given an incidence structure $D$, determine the edge domination number of its incidence graph $I(D)$.}
	\end{center} For convenience's sake, we will denote this quantity by $\gamma_e(D)$ instead of $\gamma_e(I(D))$. \\ 
	
	We prove a number of general bounds for $\gamma_e(D)$ for incidence structures throughout this paper.  We highlight one such result below, which shows that for symmetric designs $D$, determining $\gamma_e(D)$ is equivalent to determining the maximum size of certain incidence-free sets.
	
	\begin{df}\label{def:incidencefree}
		Let $X \subseteq \mp$ and $Y \subseteq \mb$ be such that there are no incidences between points in $X$ and blocks in $Y$. In this case the pair $(X,Y)$ will be called \emph{incidence-free}. If $|X| = |Y|$ we will say that the pair is \emph{equinumerous} and refer to $|X|$ as its size.
	\end{df}
	We note that the problem of determining the size of large incidence-free subsets of incidence structures is a well-studied problem, see for example \cite{dewinterschillewaertverstraete,mattheuspavesestorme,mubayiwilliford,stinson}.  As such, the following connection to $\gamma_e(D)$ and incidence-free sets is of particular interest.
	\begin{thm}
		\label{ThmMainIntro}
		Let $D=(\mp,\mb)$ be a symmetric $(v,k,\lambda)$-design with $k\ge 100$.  Then
		\[\gamma_e(D)=v-\alpha,\]
		where $\alpha$ is the maximum size of an equinumerous incidence-free pair $(X,Y)$.
	\end{thm}
	The bound $k\ge 100$ can easily be improved, but we do not optimise this value since, in particular, \Cref{EqDesigns} implies there exist only a finite number of symmetric $(v,k,\lambda)$-designs with $k < 100$.  It is easy to show that \Cref{ThmMainIntro} can not be strengthened to hold for arbitrary SIS's, but a somewhat analogous statement does hold if one assumes some extra mild conditions on an SIS; see \Cref{ThmSISMatching} for more on this.

	\subsection{Organisation}
	The rest of the paper is organised as follows. 
	The first half of the paper is dedicated to general techniques and results, with preliminaries established in \Cref{SecPrel}, lower bounds in \Cref{sec:lowerbounds}, upper bounds in \Cref{SecConstructions}, and a proof of \Cref{ThmMainIntro} in \Cref{sec:incidencefreeishall}.

	The last half of the paper, \Cref{SecParticularClasses}, establishes bounds on specific classes of incidence structures, both by applying results from the first half of the paper and by utilising algebraic structures and symmetries particular to the individual incidence structures.  Some specific classes of designs we look at include those coming from projective planes and Hadamard matrices, and in many of these cases we establish asymptotically sharp bounds.  Our examples cover almost all known basic constructions of symmetric designs (see \cite[Part II Chapter 6.8]{colbourndinitz}), with the notable exception of symmetric designs coming from difference sets, as there seems to exist no uniform way of establishing effective bounds in this case.  The problem of solving this remaining case, as well as a number of other open problems, is discussed in \Cref{SecConclusion}.
	
	\section{Preliminaries}
	\label{SecPrel}
	
	The following key observation allows us to translate upper bounds for sizes of incidence-free sets into lower bounds for $\gamma_e(D)$.  This will be the driving force behind almost all of the lower bounds throughout this paper.
	\begin{lm}\label{EdgeDomToIncidenceFree}
		If $D = (\mp,\mb)$ is an incidence structure with $|\mp|=v$ and $|\mb| = b$, then there is an incidence-free set $(X,Y)$ with $|X| \geq v-\gamma_e(D)$ and $|Y| \geq b-\gamma_e(D)$.
	\end{lm}
	\begin{proof}
		Let $\Gamma$ be an edge dominating set of size $\gamma_e(D)$.  Let $X\subseteq \mp$ denote the set of points which are not contained in any edge of $\Gamma$, and similarly define $Y\subseteq \mb$.  Note that $(X,Y)$ is incidence-free by definition of $\Gamma$ being an edge dominating set.  Since there are at most $|\Gamma|$ points of $\mp$ in an edge of $\Gamma$, we have
		\[|X|\ge |\mp|-|\Gamma|=v-\gamma_e(D),\]
		and the same argument gives the desired bound for $Y$.
	\end{proof}
	
	To construct small edge dominating sets, we rely on matchings in $I(D)$.  Recall that a \emph{matching} in a graph $G=(V,E)$ is a subset $M \subseteq E$ such that each vertex of $G$ lies in at most one edge of $M$.
	A matching is called \emph{maximal} if it cannot be extended to a larger matching.
	It is not hard to see that a matching is maximal if and only if it is an edge dominating set, and an observation of Yannakakis and Gavril~\cite{gavril} implies that $\gamma_e(G)$ is equal to the smallest size of a maximal matching of $G$.
	
	We say that a matching $M$ of $G=(V,E)$ \emph{covers} a subset $S$ of vertices if each element of $S$ is contained in an element of $M$. A \emph{perfect matching} is a matching that covers all vertices exactly once. Note that perfect matchings are maximal matchings and edge dominating sets.  A standard but important fact about matchings that we need is the following easy consequence of Hall's theorem, where we recall that a graph is \textit{biregular} if it has a bipartition $U\cup V$ such that every vertex within $U$ has the same degree and every vertex within $V$ has the same degree.
	\begin{lm}\label{LmBiregular}
		If $G$ is a biregular graph with bipartition $U\cup V$ such that $|U|\le |V|$, then $G$ has a matching which covers $U$.
	\end{lm}
	
	We can use this to show when a trivial upper bound on $\gamma_e(D)$ is tight for designs.
	
	\begin{lm}
		\label{LmLargeRepNumber}
		If $D = (\mp,\mb)$ is an incidence structure with $|\mp|=v$, then
		\[\gamma_e(D)\le v.\]
		If $D$ is a $(v,k,\lambda)$-design, then equality holds if and only if $r \geq v$.
	\end{lm}
	
	\begin{proof}
		To prove the upper bound, we can assume without loss of generality that every point $P\in \mp$ is incident to at least one block $B_p\in \mb$.  Then  $\set{(P,B_P)}_{P\in \mp}$ is an edge dominating set of $I(D)$ of size $v$, which shows $\gamma_e(D) \leq v$.  From now on we suppose $D$ is a design. 
		
		First consider the case $r\ge v$.  Let $\Gamma$ be a smallest edge dominating set of $I(D)$.  If every point in $\mp$ is contained in an edge of $\Gamma$, then $\gamma_e(D)=|\Gamma|\ge |\mp|=v$.  Thus we may assume there exists some point $P$ not contained in an edge of $\Gamma$.  Since $\Gamma$ is an edge dominating set, this implies that all of the $r$ blocks incident to $P$ are contained in an edge of $\Gamma$.  Thus $\gamma_e(D)\ge r\ge v$, proving the `if' statement.
		
		Now consider the case $r<v$, and let $\mb_P$ denote the set of $r$ blocks incident to a point $P \in \mp$.
		The incidence graph of $(\mp \setminus \set P, \mb_P)$ is biregular, hence it has a matching $\Gamma'$ covering $\mb_P$ by \Cref{LmBiregular}.  Since $|\Gamma'|=r<v$, one can find a set of edges $\Gamma$ which contains $\Gamma'$ and is such that each point of $\mp\setminus \set P$ is contained in exactly one edge of $\Gamma$ (namely, by arbitrarily adding to $\Gamma$ an edge containing each point not covered by $\Gamma'$).  This is an edge dominating set of size $v-1$, completing the proof of the if and only if statement.
	\end{proof}
	
	We note that the equality case of \Cref{LmLargeRepNumber} partially motivates our focus on studying symmetric incidence structures where the number of points is comparable to the number of blocks.  However, even in this setting the trivial upper bound of \Cref{LmLargeRepNumber} is often close to the true value. This is perhaps not too surprising given that the incidence graphs of designs are `expanding' in the sense of Thomason \cite[Section 3.3]{thomason}. Roughly speaking, this means that small sets of points are incident with a relatively large number of blocks and vice versa. This implies that edges are `well-spread' and do not concentrate on small sets of vertices. It is thus reasonable to think that in order to dominate all edges, we need a fairly large number of them. This is indeed the case as we will see.

	\section{Lower bounds}\label{sec:lowerbounds}
	To start this section, we will give a lower bound on the number of blocks incident with at least one point of a certain point set.
	For designs, this can be done using eigenvalue techniques, as was done by Haemers \cite[Corollary 5.3]{haemers}. We will revisit this idea in \Cref{SubsecSemibiplanes}.
	
	This same result can also be proved using simple counting techniques.
	Stinson \cite[Theorem 3.1]{stinson} and De Winter, Schillewaert, Verstraëte \cite[Theorem 3]{dewinterschillewaertverstraete} independently bound the largest number of non-incident points and blocks in a projective plane.
	Later, Elvey Price, Adib Surani, Zhou \cite[Lemma 3]{elveyprice} gave a combinatorial proof for Haemers' bound for symmetric designs.
	However, their arguments can be generalised in a straightforward way to hold for all incidence structures of type $(v,b,k,r,\lambda)$.
	For the sake of completeness, we give such a proof.
	
	\begin{df}
		A \emph{maximal arc (of order n)} in an incidence structure of type $(v,b,k,r,\lambda)$ is a non-empty set of points $S$ such that every block intersects $S$ in either 0 or $n$ points, for some integer $n$.
		We call a maximal arc \emph{trivial} if it consists of one point, all points, or the complement of a block\footnote{The complement of a block $B$ is a maximal arc if and only if every other block intersects $B$ in a constant number of points.
			This is the case in symmetric designs.}.
	\end{df}
	
	Suppose that $D$ is a $(v,k,\lambda)$-design, and that $S$ is a maximal arc of order $n$ in $D$.
	Fix a point $P \in S$.
	Then there are $r$ blocks incident to $P$, each containing $n-1$ other points of $S$.
	Since each point of $S \setminus \set P$ is incident with $\lambda$ blocks incident to $P$, this implies that $|S| = 1 + r \frac{n-1}\lambda$, or equivalently that the order of a maximal arc $S$ equals $1+\frac{|S|-1}{r}\lambda$.

	Now take a point $Q \notin S$.
	Consider the set $\sett{(P,B) \in S \times \mb}{P,Q \in B}$.
	By performing a double count on this set, we see that there are $(1 + r \frac{n-1} \lambda) \frac \lambda n = r - \frac{r-\lambda}n$ blocks incident with $Q$ containing $n$ points of $S$.
	So, if $T$ denotes the set of blocks not incident with any point of $S$, then any point is incident with either 0 or $\frac{r-\lambda}n$ blocks of $T$.
	We call $T$ the \emph{dual arc} of $S$.
	
	\begin{lm}
		\label{LemBlocksThroughPoints}
		Let $D = (\mp,\mb)$ be an incidence structure of type $(v,b,k,r,\lambda)$.
		For any set $S \subseteq \mp$, there are at least $\frac{r^2|S|}{r + (|S| - 1)\lambda}$ blocks which intersect $S$.
		Equality holds in designs if and only if $S$ is a maximal arc of order $1 + \frac{|S|-1} r \lambda$.
	\end{lm}
	
	\begin{proof}
		Let $x$ be the number of blocks intersecting $S$, and let $n_i$ be the number of blocks intersecting $S$ in exactly $i$ points.
		Then the so-called standard equations tell us that
		\begin{align*}
			\sum_{i>0} n_i = x, &&
			\sum_{i>0} i n_i = r|S|, &&
			\sum_{i>0} i^2 n_i \leq |S|(r + (|S|-1) \lambda),
		\end{align*}
		by performing a double count on $\sett{ (P,B) \in S \times \mb}{P \in B}$ and $\sett{ (P,Q,B) \in S \times S \times \mb}{ P,Q  \in B}$. Using these equations we find
		\begin{align*}
			0 & \leq \sum_{i>0} \left(i - \left( 1 + \frac{|S|-1} r \lambda \right) \right)^2 n_i \\
			&= \sum_{i>0} i^2 n_i - 2 \left( 1 + \frac{|S|-1} r \lambda \right) \sum_{i>0}i n_i + \left( 1 + \frac{|S|-1} r \lambda \right)^2 \sum_{i>0} n_i \\
			& \leq |S|(r + (|S|-1) \lambda) - 2 \left( 1 + \frac{|S|-1} r \lambda \right) r|S| + \left( 1 + \frac{|S|-1} r \lambda \right)^2 x \\
			\iff & \left( 1 + \frac{|S|-1} r \lambda \right)^2 x  \geq 2 \left( 1 + \frac{|S|-1} r \lambda \right) r |S|  - |S|(r + (|S|-1) \lambda) \\
			%= \left( 1 + \frac{|S|-1} r \lambda \right)(2|S| r - |S| r) \\
			\iff & \,\,x  \geq \frac{r|S|}{1 + \frac{|S|-1} r \lambda} = \frac{r^2|S|}{r + (|S|-1) \lambda}.
		\end{align*}
		Equality holds if and only if $n_i$ is only non-zero for $i = 0$ and $i = 1 + \frac{|S|-1} r \lambda$ and any two points of $S$ determine $\lambda$ blocks, which is always true in designs.
		This is equivalent to $S$ being a maximal arc of order $1 + \frac{|S|-1} r \lambda$.
	\end{proof}
	
	We can do a bit better for small values of $|S|$, which we will need later on.

	\begin{lm}
		\label{LmIneqCap}
		Let $D = (\mp,\mb)$ be an incidence structure of type $(v,b,k,r,\lambda)$.
		For any set $S \subseteq \mp$, there are at least $r|S| - \lambda \binom{|S|}{2}$ blocks which intersect $S$.
		Equality holds in designs if and only if no three (or more) points of $S$ are incident with the same block.
	\end{lm}
	
	\begin{proof}
		
		We can recycle the definition of $x$, $n_i$ and the inequalities above to compute
		
		\begin{align*}
			0 &\leq \sum_{i > 0}(i-1)(i-2)n_i \\
			& = \sum_{i>0} i^2 n_i - 3 \sum_{i>0}i n_i + 2 \sum_{i>0} n_i \\
			& \leq |S|(r+(|S|-1)\lambda)-3r|S|+2x \\
			&\Leftrightarrow x \geq r|S|-\lambda\binom{|S|}{2}
		\end{align*}	
		
		In case of equality, we see that only $n_1$ and $n_2$ can be non-zero, which implies the last statement.
	\end{proof}
	
	One can check that Lemma \ref{LmIneqCap} yields a better bound than Lemma \ref{LemBlocksThroughPoints} if and only if $|S| < 1 + \frac r \lambda $. In fact, a set of points with no three in a block can have at most $1 + r/\lambda$ points, and equality holds if and only if it is a maximal arc of order 2. In that case, both bounds coincide.
	
	We note that in \Cref{sec:incidencefreeishall} we will similarly want to estimate the number of blocks incident with a small set of points. There we will use a greedy approach, which would work equally well to prove the lemma above. \\
	
	We will now use the first bound combined with \Cref{EdgeDomToIncidenceFree} to obtain lower bounds on the edge domination number of designs. 
	We restrict ourselves to designs instead of general incidence structures both for ease of presentation and the fact that we do not need the general bounds in the remainder of this text.
	
	\begin{prop}
		\label{PropUpperBoundIF}
		Let $D = (\mp,\mb)$ be a $(v,k,\lambda)$-design and let $(X,Y)$ be an incidence-free pair satisfying $v - |X| = b - |Y|$. Then
		\begin{enumerate}[(1)]
			\item \( \displaystyle |X| \leq r \frac {k-r-2 + \sqrt{(r-k)^2 + 4(r-\lambda)}}{2\lambda} + 1 \). 
			
			Equality holds if and only if $X$ is a maximal arc of order $\displaystyle \frac{k - r + \sqrt{(r-k)^2 + 4(r-\lambda)}}2$ and $Y$ is its dual arc.
			
			\item If $D$ is a symmetric design, then \( \displaystyle |X| \leq k \frac{\sqrt{k-\lambda} - 1}{\lambda} + 1 \). 
			
			Equality holds if and only if $X$ is a maximal arc of order $\sqrt{k-\lambda}$ and $Y$ is its dual arc.
		\end{enumerate}
	\end{prop}
	
	\begin{proof}
		(1) Since $Y$ contains no blocks incident to $X$, there are at most $b - |Y| = v- |X|$ blocks incident with a point in $X$.
		By Lemma \ref{LemBlocksThroughPoints}, this implies that
		\[
		\frac{r^2 |X|}{r + (|X|-1)\lambda} \leq v - |X|.
		\]
		This gives us an equation of the form $f(|X|) \leq 0$, where $f$ is a quadratic polynomial with positive leading coefficient.
		Therefore, $|X|$ is at most the largest root of $f$.
		Calculating this root is tedious, but straightforward, and yields the bound from the Proposition, relying on \Cref{EqDesigns}.
		
		Now suppose that equality holds.
		Then we attain equality in the bound from \Cref{LemBlocksThroughPoints}, which implies that $X$ is a maximal arc.
		In this case, one can calculate from the size of $X$ that it is a maximal arc of order $(k - r + \sqrt{(r-k)^2 + 4(r-\lambda)}) / 2$.
		Necessarily, $Y$ must consist of all blocks missing $X$, i.e.\ $Y$ must be the dual arc of $X$.
		
		(2) This follows immediately from (1) using $k=r$.
	\end{proof}
	
	\begin{crl}
		\label{CrlDesignDomIneqs}
		Let $D=(\mp,\mb)$ be a $(v,k,\lambda)$-design, then
		\begin{enumerate}
			\item[(1)] $ \displaystyle\gamma_e(D) \geq  r \frac{r + k - \sqrt{ (r-k)^2 + 4(r-\lambda)}}{2 \lambda}$.
			
			Equality holds if and only if $D$ has a maximal arc of order $\displaystyle \frac{k - r + \sqrt{(r-k)^2 + 4(r-\lambda)}}2$.
			\item[(2)] If $D$ is a symmetric design, then $\gamma_e(D) \geq \displaystyle k\frac{k - \sqrt{k-\lambda}} \lambda$. 
			
			Equality holds if and only if $D$ has a maximal arc of order $\sqrt{k-\lambda}$.
		\end{enumerate}
	\end{crl}
	
	\begin{proof}
		By \Cref{EdgeDomToIncidenceFree}, there exists an incidence-free pair $(X,Y)$ with $|X| \geq v-\gamma_e(D)$, and $|Y| \geq b-\gamma_e(D)$.
		Hence, we can find subsets $X' \subseteq X$ and $Y' \subseteq Y$ of size $v-\gamma_e(D)$ and $b-\gamma_e(D)$ respectively.
		Note that $(X',Y')$ is necessarily also incidence-free.
		Plugging this into \Cref{PropUpperBoundIF} yields the desired bound.
		
		Vice versa, suppose that $D$ has a maximal arc $X$ of order $n = (k - r + \sqrt{(r-k)^2 + 4(r-\lambda)})/2$.
		Let $Y$ denote its dual arc.
		Consider the induced subgraph $H$ of $I(D)$ on $(\mp \setminus X) \cup (\mb \setminus Y)$.
		Then every vertex in $\mp \setminus X$ has degree $r - \frac{r-\lambda}n$, and every vertex in $Y$ has degree $k-n$, i.e.\ $H$ is biregular.
		Furthermore, from equality in \Cref{PropUpperBoundIF}, it follows that $|\mp \setminus X| = |\mb \setminus Y|$.
		Thus, $H$ is a regular bipartite graph, and hence has a perfect matching by \Cref{LmBiregular}.
		Since there are no edges between $X$ and $Y$ in $I(D)$, this matching is an edge dominating set of size $v - |X|$.
		It is straightforward to check that $v-|X|$ equals the right-hand side of the above bound.
	\end{proof}

	\begin{rmk}
		\label{RmkBoundDomSIS}
		\begin{enumerate}
			\item Recall that for a symmetric $(v,k,\lambda)$-design we have $r=k$ and $v = k(k-1)/\lambda + 1$, so (2) can be restated as \[\gamma_e(D) \geq v - \frac{k\sqrt{k-\lambda}-k}{\lambda}-1.\] This indeed shows that $\gamma_e(D)$ is equal to $v$ up to lower order terms.
			\item Laskar and Wallis \cite{laskar} showed that if $D$ is a symmetric $(n^2+n+1,n+1,1)$ design one has $\gamma_e(D) > \half (n^2+3n)$. Our results improve this to $\gamma_e(D) \geq n^2-n\sqrt{n}+2n-\sqrt{n}+1$. 
			\item One could try to improve the bounds above by studying the existence of maximal arcs in designs. This is however not feasible to do in general: even for the class of projective planes defined over a finite field $\FF_q$, which are a special type of symmetric $(q^2+q+1,q+1,1)$-designs, the existence of maximal arcs depends on the parity of $q$ as we will see later. Improvements are thus only possible by restricting oneself to particular classes of designs.
		\end{enumerate}
		
	\end{rmk}

	\section{Constructing large incidence-free sets}\label{SecConstructions}
	In view of Theorem~\ref{ThmMainIntro}, upper bounding $\gamma_e(D)$ when $D$ is a symmetric design is equivalent to constructing large incidence-free pairs.  Here we discuss two general techniques for achieving this goal.
	
	\subsection{A probabilistic construction}
	
	In a symmetric $(v,k,\lam)$-design, one can greedily construct sets $X\sub \mp,\ Y\sub \mb$ such that no point in $X$ is contained in a block of $Y$ and such that $|X|=|Y|\ge v/(k+1)\approx k/\lam$.  One can improve this bound somewhat through a simple probabilistic argument, which is based on a similar proof by De Winter, Schillewaert, and Verstraëte \cite[Theorem 6]{dewinterschillewaertverstraete}.
	\begin{prop}\label{PropProbability}
		Let $D=(\mp,\mb)$ be an SIS of type $(v,k,\lam)$ with $\lam\le \frac{k}{64\log k}$.  
		Then there exists an equinumerous incidence-free pair whose size is at least $\frac{k \log k}{2 \lambda}$.
	\end{prop}
	
	\begin{proof}
		Let $X\sub \mp$ be a random set of points obtained by including each point in $\mp$ independently with probability $p=\frac{\log k}{2k}$, and let $Y\sub \mb$ consist of the blocks which do not contain any points of $X$.  Note that by construction no point in $X$ is contained in a block of $Y$, so it suffices to show that with positive probability $|X|,|Y|\ge \frac{k\log k}{\lambda}$.
		
		First observe that \begin{equation}\E[|X|]=p v\ge \frac{\log k}{2k}\cdot \frac{k(k-1)}{\lam}\ge \frac{k \log k}{4\lam } \ge 16,\label{EqXExpectation}\end{equation}
		where the second to last step used $k\ge 2$, which is implicit from the bounds $1\le \lambda\le \frac{k}{64\log k}$, and the last step used $\lam\le \frac{k}{64\log k}$.  Since $|X|$ is a binomial random variable, by the Chernoff bound we have
		\begin{equation}\label{EqXConcentration}
			\Pr\left[|X|\le \frac{1}{2}\E[|X|]\right]\le e^{-\E[|X|]/8}< \frac12.
		\end{equation}
		
		For $B\in \mb$, let $A_B$ denote the event that $B\in Y$.  Then
		\[\Pr[A_B]=(1-p)^k\ge e^{-kp/(1-p)}\ge k^{-1/4} ,\]
		where this first inequality follows from exponentiating the  inequality $ \log(1+x)\ge \frac{x}{1+x}$, and the last used $1-p\ge 1/2$.  As $v \geq \frac{k(k-1)}{\lambda}+1$, we conclude that
		\begin{equation}\E[|Y|]=\sum_{B\in \mb} \Pr[A_B]\ge \frac{k\log k}{2\lam}.\label{EqYExpectation}\end{equation}
		Note that for any $B,B'\in \mb$ (possibly non-distinct), we have
		\[\Pr[A_B\cap A_{B'}]\le (1-p)^{2k-\lam}=(1-p)^{-\lam} \Pr[A_B]\Pr[A_{B'}]\le e^{\lam p/(1-p)}\Pr[A_B]\Pr[A_{B'}]\le e^{1/32} \Pr[A_B]\Pr[A_{B'}],\]
		where the second inequality used $\log(1+x)\ge \frac{x}{1+x}$ as before, and the second used $\lam p\le 1/64$ and $1-p\ge 1/2$ by hypothesis.  Combining these two results gives
		\[\mathrm{Var}(|Y|)=\sum_{B,B'\in \mb} \Pr[A_B\cap A_{B'}]-\Pr[A_B]\Pr[A_B']\le \left(e^{1/32}-1\right) \E[|Y|]^2\le .04 \E[|Y|]^2.\] 
		
		Using this, Chebyshev's inequality and denoting $Z = |Y|$ for readability, we have
		\begin{equation}
			\Pr \left(Z\le \frac{1}{2} \E[Z] \right)\le
			\Pr \left[ |Z-\E[Z]| \geq \half \E[Z] \right] \le
			\frac{\mathrm{Var}(Z)}{\frac{1}{4}\E[Z]^2}\le \frac{4}{25}.\label{EqYConcentration}
		\end{equation}
		
		By \eqref{EqXExpectation}, \eqref{EqXConcentration}, \eqref{EqYExpectation} and \eqref{EqYConcentration} we conclude with positive probability we have $|X|,|Y|\ge \frac{k \log k}{2\lam } $, and hence there exists some pair of sets $X,Y$ of at least this size such that no point in $X$ is contained in a block of $Y$.
		Taking subsets of $X,Y$ of size $\min\{|X|,|Y|\}\ge \frac{k \log k}{2\lam }$ gives the result.
	\end{proof}

	For symmetric designs, this probabilistic construction is suboptimal for proving asymptotically sharp bounds on the edge domination number, and it is reasonable to think that more advanced techniques might provide better constructions.
	However, \cite[Problem 7]{dewinterschillewaertverstraete} suggests that improving this result to $k^{1+\eps}/\lam$ for any $\eps > 0$ might be hard when $\lambda$ is small, as it is related to an open and difficult conjecture by Erd\H{o}s on $C_4$-free graphs and its generalisation to $K_{2,t}$-free graphs.
	
	This result (and its limitations mentioned above) motivates the need to look at particular classes of symmetric designs in order to do better, which is exactly what we will do in Section~\ref{SecParticularClasses}. %The families we will consider are those coming from projective planes, Hadamard matrices and related designs. For these families we will show the (asymptotic) sharpness of our bounds. For the large class of symmetric designs coming from difference sets, it is not quite clear how to obtain uniform constructions of incidence-free sets, i.e.\ without having to consider the specific features of each difference set separately. These classes cover almost all known basic constructions of symmetric designs \cite[Part II Chapter 6.8]{colbourndinitz}.

	\subsection{From polarities}
	
	In this section we describe a technique to construct equinumerous incidence-free pairs in symmetric designs with special kinds of symmetries, such as the Hadamard designs and those corresponding to projective planes over finite fields.
	
	\begin{df}
		A \emph{polarity} of a symmetric design $D = (\mp,\mb)$ is a bijection $\rho: \mp \cup \mb \to \mp \cup \mb$, such that
		\begin{enumerate}
			\item $\rho$ maps points to blocks and blocks to points,
			\item $\rho$ preserves incidence: for each pair $(P,B) \in \mp \times \mb$ it holds that $P\in B$ if and only if $\rho(B) \in \rho(P)$,
			\item $\rho^2$ is the identity map.
		\end{enumerate}
		A point or block is \emph{absolute} if it is incident with its image under $\rho$.
	\end{df}
	
	Given a symmetric design $D = (\mp, \mb)$ with a polarity $\rho$, define its 
	\emph{polarity graph} $R(D,\rho)$ as the graph with $\mp$ as vertices, where two vertices $P$ and $Q$ are adjacent if and only if $P \in \rho(Q)$. Note that each absolute point will give rise to a loop in this graph.  With this we can translate the problem of finding equinumerous incidence-free pairs to finding cocliques in the polarity graph, which we record in the following observation.
	
	\begin{lm}\label{LemIncidencFreeFromPolarity}
		If $C$ is a coclique in $R(D,\rho)$, then $(C,\rho(C))$, where $\rho(C) := \sett{\rho(P)}{P \in C}$, is an equinumerous incidence-free pair.
	\end{lm}
	
	We count loops as edges, which implies that a coclique in the polarity graph contains no absolute points.

	\section{Incidence-free sets in symmetric designs}\label{sec:incidencefreeishall}
	In this section we prove \Cref{ThmMainIntro}, which says that for symmetric designs $D$, the edge domination number $\gamma_e(D)$ is equal to $v-\alpha$ where $\alpha$ is the size of a largest equinumerous incidence-free pair $(X,Y)$.  The fact that $\gamma_e(D)$ is at least this quantity follows from \Cref{EdgeDomToIncidenceFree}, so it remains to construct an edge dominating set of this size.  We do this through the following.
	\begin{prop}\label{PropHallPair}
		Let $D=(\mp,\mb)$ be a symmetric $(v,k,\lam)$-design with $k\ge 100$. If $(X,Y)$ is an equinumerous incidence-free pair, then there is a perfect matching between $\mp \setminus X$ and $\mb \setminus Y$ in $I(D)$.
	\end{prop}
	The perfect matching guaranteed by \Cref{PropHallPair} gives an edge dominating set of size $v-|X|$, so choosing a largest equinumerous incidence-free pair gives the upper bound of \Cref{ThmMainIntro}.  Thus to prove \Cref{ThmMainIntro}, it only remains to prove \Cref{PropHallPair}.  We do this in the next subsection.  
	
	\subsection{\texorpdfstring{Proof of \Cref{PropHallPair}}{Proof of Proposition 5.1}}
	For the remainder of this subsection, we fix the following assumptions and notation:
	\begin{itemize}
		\item $D = (\mp,\mb)$ is a symmetric $(v,k,\lambda)$-design with $k \geq 100$;
		\item $(X,Y)$ is an equinumerous incidence-free pair in $D$;
		\item $G$ will denote the induced subgraph of $I(D)$ obtained by removing $X$ and $Y$ from its vertex set;
		\item $\alpha := |X| = |Y|$.
	\end{itemize}	
	We note that under the assumption $k\ge 100$, we have
	\begin{equation}\label{EqAlpha}
		\al\le  \frac{k^2}{10\lam}.
	\end{equation}
	Indeed, by \Cref{PropUpperBoundIF} we have $\al\le \frac{k^{3/2}}{\lam} \le \frac{k^2}{10\lam}$, where this last step used $k\ge 100$.
	
	We will use the following easy consequence of Hall's theorem to show the existence of perfect matchings in $G$; see e.g.\  \cite{ehrenhorg} for details of its proof.  Here $N_{G'}(S)$ denotes the set of vertices which are adjacent to some vertex of $S$ in a graph $G'$.
	\begin{lm}\label{LmHallVariant}
		Let $G'$ be a bipartite graph on $U\cup V$ with $|U|=|V|$.  If every $S\sub U$ and $T\sub V$ with $|S|,|T|\le \ceil{\half |U|}$ satisfies $|N_{G'}(S)|\ge |S|$ and $|N_G(T)|\ge |T|$, then $G'$ has a perfect matching.
	\end{lm}
	
	We first show that \Cref{LmHallVariant} is satisfied  whenever $S\sub \mp \setminus X$ is relatively large. For this, we observe that by definition of $D$ being a symmetric $(v,k,\lambda)$-design, the graph $I(D)$ is $K_{2,\lambda+1}$-free, and hence all of its subgraphs are as well.  By using the famous result by K\H{o}v\'ari, S\'os, and Tur\'an~\cite{kovarisosturan}, which upper bounds the maximum number of edges that a $K_{s,t}$-free bipartite graph can have, we immediately get the following.
	\begin{lm}
		%Let $D = (\mp,\mb)$ be a $2$-$(v,k,\lambda)$ design. 
		For any $S \subseteq \mp$ and $T \subseteq \mb$, we have
		\[e(S,T) \leq \sqrt{\lambda}(|S|-1)\sqrt{|T|} + |T|,\]
		where $e(S,T)$ denotes the number of edges in $I(D)$ between $S$ and $T$.
	\end{lm}
	This gives the following.
	\begin{lm}\label{LemKST}
		If $S\sub \mp\setminus X$ satisfies \[\lam (|S|+\al)\le \left(k -1-\frac{\al}{|S|}\right)^2,\] then $|N_{G}(S)|\ge |S|$.
	\end{lm}
	\begin{proof}
		Let $S$ be as in the lemma statement and set $T=N_{I(D)}(S)$.  If $|T|\ge |S|+\al$, then
		\[|N_{G}(S)|=|T\sm Y|\ge |T|-|Y|=|T|-\al\ge |S|.\]
		Thus we may assume for contradiction that $|T|<|S|+\al$.  
		%Observe that $G$ contains no copy of $K_{2,\lam+1}$ since every pair of vertices has at most $\lam$ common neighbours.  This continues to hold for the induced subgraph $G[S\cup T]$, so by the K\H{o}v\'ari-S\'os-Tur\'an theorem and the fact that every vertex in $S$ has exactly $k$ neighbours in $T$, we have
		By the preceding lemma, we find
		\[k|S|=e(S,T)\le \sqrt{\lambda}(|S|-1)\sqrt{|T|}+|T|<\sqrt{\lambda}|S|\sqrt{|S|+\al}+|S|+\al.\]
		By dividing both sides by $|S|$, we see that the above is equivalent to
		\[\sqrt{\lambda(|S|+\al)}>k-1-\frac{\al}{|S|}.\]
		This contradicts our choice of $S$, so we conclude the result.
	\end{proof}
	\begin{crl}\label{CrlLarge}
		If $S\sub \mp\setminus X$ satisfies 
		\[\al \le |S|\le \frac{(k-2)^2}{\lam}-\al,\]
		then $|N_{G}(S)|\ge |S|$.
	\end{crl}
	Note that the number of vertices in each part of $G$ is exactly $\frac{k(k-1)}{\lam}+1-\al$, so this corollary shows that Hall's condition is satisfied for almost all sets of size at least $\al$. 
	\begin{proof}
		Let $S$ be as in the hypothesis of the statement and assume for contradiction that
		\[\lam (|S|+\al)> \left(k -1-\frac{\al}{|S|}\right)^2\ge (k-2)^2,\]
		where this last inequality used $|S|\ge \al$. This contradicts our assumption on $S$.  We conclude that the hypothesis of \Cref{LemKST} is satisfied, and hence $|N_{G}(S)|\ge |S|$.
	\end{proof}
	\begin{crl}\label{CrlMedium}
		If $S\sub \mp\setminus X$ satisfies
		\[\frac{2\al}{k}\le |S|\le \al,\]
		then $|N_{G}(S)|\ge |S|$.
	\end{crl}
	\begin{proof}
		Let $S$ be as in the hypothesis of the statement and assume for contradiction that 
		\[\lam (|S|+\al)> \left(k -1-\frac{\al}{|S|} \right)^2\ge \left(\half k-1\right)^2 \ge \frac{1}{5} k^2,\]
		where the second inequality used $\frac{\al}{|S|} \le \half k$, and the last inequality used $k\ge 100$.  Because $|S|\le \al$, this implies $2\lam \al > \frac{1}{5} k^2$, and this does not hold by \eqref{EqAlpha}.  We conclude that the hypothesis of \Cref{LemKST} is satisfied, and hence $|N_{G}(S)|\ge |S|$.
	\end{proof}
	Note that up to this point, we have not used the hypothesis that $(X,Y)$ is incidence-free.  We use this to deal with sets of size at most $\frac{2\al}{k}$. The following lemma can be seen as a variation of \Cref{LmIneqCap}.
	\begin{lm}\label{LmSmall}
		Every $S\sub \mp\setminus X$ has
		\[|N_{G}(S)|\ge \max_{s\le |S|} \frac{\lam \al s}{k} -\lam \binom s 2.\]
	\end{lm}
	\begin{proof}
		Fix any vertex $P\in \mp\setminus X$.  Observe that each vertex in $X$ has $\lam$ common neighbours with $P$ in $I(D)$, and that these neighbours necessarily lie in $\mb\setminus Y$ since $X$ and $Y$ are incidence-free.  Thus 
		\[\lam \al=\lam |X|\le e_{I(D)}(N_{G}(P), X)\le k|N_{G}(P)|,\]
		where this last step used that each vertex of $I(D)$ has degree $k$.  We conclude that $|N_{G}(P)|\ge \frac{\lam \al}{k}$ for all $P\in \mp\setminus X$.
		
		Now given any $S\sub \mp\setminus X$, let $S'\sub S$ be a subset of size $s$.  By an inclusion-exclusion argument, we see
		\[|N_{G}(S)|\ge |N_{G}(S')|\ge \sum_{P\in S'} |N_{G}(P)|-\sum_{P,Q\in S'} |N_{G}(P)\cap N_{G}(Q)|\ge \frac{\lam \al s}{k}-\lam \binom s 2,\]
		where this last step used the previous observation and that every two distinct vertices in $S'$ have at most $\lam$ common neighbours in $G$.  The above inequality holds for any $s\le |S|$, giving the result.
	\end{proof}

	We now have everything in place to prove our main result for this subsection.
	\begin{proof}[Proof of \Cref{PropHallPair}]
		%By the first lemma of this section, we see that it suffices to prove $G=I(D)-X-Y$ has a perefect matching.  
		If $\al=0$ then $G=I(D)$ is a regular bipartite graph, which has a perfect matching by \Cref{LmBiregular}.  Thus we may assume $\al>0$.  We now show that $G$ has a perfect matching by using \Cref{LmHallVariant}, and by the symmetry of the problem, it suffices to prove $|N_{G}(S)|\ge |S|$ whenever $S\sub \mp\setminus X$ with $|S|\le \ceil{\half |\mp\setminus X|}$.
		
		We first claim that $\frac{(k-2)^2}{\lam}-\al\ge \ceil{\half |\mp \setminus X|}$.  Indeed, recall that $|\mp|=\frac{k(k-1)}{\lam}+1\le \frac{k^2}{\lam}$.  Thus to prove the claim it suffices to prove
		\[2 \frac{(k-2)^2}{\lam}-2\al\ge \frac{k^2}{\lam}+1-\al,\]
		where the extra $1$ stems from upper bounding the ceiling function.  This is equivalent to showing
		\[\al \le \frac{k^2-8k+8-\lam}{\lam}.\]
		Note that $k^2-8k+8-\lam \ge (k-9)k\ge \frac{1}{10} k^2$ since $k\ge \lambda$ and $k\ge 100$.  Thus the inequality above follows from \eqref{EqAlpha}, proving the claim.
		
		With this claim, Corollaries~\ref{CrlLarge} and \ref{CrlMedium} imply $|N_{G}(S)|\ge |S|$ for $\frac{2\al}{k}\le |S|\le \ceil{\half  |\mp\setminus X|}$.  If $\lam\ge 2$ and $|S|\le \frac{2\al}{k}$, then Lemma~\ref{LmSmall} with $s=1$ implies $|N_{G}(S)|\ge |S|$, proving the result in this case.
		If $\lam=1$ and $2\le |S|<\frac{2\al}{k}$, then $|S| \leq \ceil{\frac{2 \alpha}k-1}$, and Lemma~\ref{LmSmall} with $s=2$ implies $|N_{G}(S)|\ge \ceil{\frac{2 \alpha}k-1} \geq |S|$.  Finally, if $\lambda =1$ and $S=\{P\}$, we claim there is at least one block in $N_{I(D)}(P) \setminus Y$. Indeed, $X$ is non-empty since $\al>0$, and any $Q\in X$ is contained in a unique block $B$ with $P$ since $\lambda=1$.
		We can not have $B\in Y$ since $(X,Y)$ is incidence-free, so $B\in N_{I(D)}(P) \setminus Y$. We conclude that the conditions of \Cref{LmHallVariant} are satisfied in all cases, giving the result.
	\end{proof}
	
	\subsection{Adapting the proof to symmetric incidence structures}
	As mentioned in the introduction, \Cref{ThmMainIntro} does not hold for SIS's in general.  Indeed, if $D$ is the disjoint union of two symmetric designs $(\mp_1,\mb_1)$ and $(\mp_2,\mb_2)$, then taking $X=\mp_1$ and $Y=\mb_2$ shows that such a result can not hold in general.  Counterexamples continue to exist even if one assumes $I(D)$ is connected, as we will see in \Cref{RmkConnectedCounterexample}.
	
	This being said, there were only a few crucial instances where we used the properties of a design. Firstly, we used \Cref{PropUpperBoundIF} to show that $\alpha$ is small in comparison to $v$. Secondly, in \Cref{LmSmall} we used that in a design every two points have at least $\lambda$ common neighbours in order to show small sets satisfied  Hall's condition. In this subsection, we will replace these consequences by making assumptions on the size of $X$ and a minimum degree condition in $G$, respectively. 
	
	In addition to this, we will require that the incidence structure is somewhat close to a symmetric design in the sense that $v$ is comparable to $\frac{k^2}{\lambda}$. The construction indicated above to obtain an SIS from disjoint unions of symmetric designs can be iterated an arbitrarily number of times. The edge density of the resulting SIS will be very small, and in this case the upper bound in \Cref{PropUpperBoundIF} will be useless. For this reason we will require that $v = c\frac{k^2}{\lambda}$ for some constant $c > 0$. 
	
	All of these conditions will allow us to construct small edge dominating sets whenever the incidence structure is not a symmetric design in \Cref{SubsecSemibiplanes}. The theorem we state will not be best possible, but suffices for our purposes in this paper and allows us to recycle the proofs in the previous subsection. 
	
	\begin{thm}
		\label{ThmSISMatching}
		Let $D=(\mp,\mb)$ be an SIS of type $(v,k,\lambda)$ where $v \leq \frac{9}{5}\frac{k^2}{\lambda}$ and $k \geq 100$.
		If $(X,Y)$ is an equinumerous incidence-free pair with $|X|\leq \frac{k^2}{10\lambda}$ and every point or block not in $X \cup Y$ is incident with at least $\frac{2|X|}{k}$ blocks or points not in $X \cup Y$, then
		\[\gamma_e(D)\le v-|X|.\]
	\end{thm}
	
	\begin{proof}
		The proof is nearly identical to that of \Cref{ThmMainIntro}, so we omit some of the redundant details.  As before, it suffices to show there is a perfect matching between $\mp \setminus X$ and $\mb \setminus Y$ in $I(D)$. 
		
		Let $G$ denote the induced subgraph of $I(D)$ after removing the vertices of $X \cup Y$ and define $\alpha := |X| = |Y|$.  It is not difficult to show that the proofs of \Cref{CrlLarge} and \Cref{CrlMedium} continue to hold in this setting by our assumptions on $\alpha$ and $k$. Instead of \Cref{LmSmall}, we find 
		\begin{align}\label{mindegree}
			|N_{G}(S)|\ge \max_{s\le |S|} \frac{2 \al s}{k} -\lam \binom s 2
		\end{align}
		for every $S \subset \mp \setminus X$ by the minimum degree condition.	
		
		Then we prove again the claim that $\frac{(k-2)^2}{\lambda}-\alpha \geq \ceil{\half|\mp \setminus X|}$. This will follow from the inequality
		\[2 \frac{(k-2)^2}{\lam}-2\al\ge \frac{9k^2}{5\lam}-\al,\]
		using the assumption $|\mp| = v \leq \frac{9k^2}{5\lambda}$. Equivalently
		\[\al \leq \frac{11k^2-80k+80}{10\lambda}.\]
		Since $\alpha \leq \frac{k^2}{10\lambda}$ by assumption, we see that the claim is satisfied for $k \geq 100$.
		
		Now we conclude the proof by observing that for all sets $S \subseteq \mp \setminus X$ with $\frac{2\alpha}{k} \leq |S| \leq \half |\mp \setminus X|$ we have $N_{G}(S) \geq |S|$ by the analogues of \Cref{CrlLarge} and \Cref{CrlMedium}. For $|S| \leq \frac{2\alpha}{k}$, we see that the minimum degree condition provides $N_{G}(S) \geq |S|$ by setting $s = 1$ in \Cref{mindegree}.
	\end{proof}

	\section{Application to particular classes of designs}\label{SecParticularClasses}
	
	\subsection{Projective point-subspace designs}
	
	In this section we will discuss the edge domination number of the design of points and $k$-spaces in the projective space $\pg(n,q)$, with $0 < k < n$. By $k$-spaces we refer to subspaces of projective dimension $k$. The $0$-, $1$-, and $(n-1)$-spaces are referred to as points, lines and hyperplanes. This design is symmetric only when $k = n-1$.
	Denote the number of points in $\pg(n,q)$ by
	\[
	\theta_n = \frac{q^{n+1}-1}{q-1}.
	\]
	
	The smallest case is the design of points and lines in the Desarguesian projective plane. In this setting, we can construct large incidence-free sets from polarities. 
	
	Assign coordinates $(x_0,\dots,x_n) \in \FF_q^{n+1} \setminus \set{(0,\dots,0)}$ to the points in $\pg(n,q)$, where $(x_0,\dots,x_n)$ and $(y_0,\dots,y_n)$ are considered as the same point if they are scalar multiples of each other.
	For each projective point $P = (a_0,\dots,a_n)$, let $P^\perp$ denote the hyperplane with equation $a_0 X_0 + \dots + a_n X_n = 0$.
	Then mapping a point $P$ to the hyperplane $P^\perp$ and vice versa, determines a polarity $\perp$ of $\pg(n,q)$.
	%It is well known that if $n=2$, this polarity has $q+1$ absolute points, which constitute a conic.
	%This conic is non-degenerate if and only if $q$ is odd.
	When $n=2$, the following lower bounds on the independence number of the polarity graph $R(D,\rho)$ are known.
	
	\begin{thm}[{\cite{mubayiwilliford,mattheuspavesestorme}}]\label{cocliquepolaritygraph}
		Let $D$ be the design of points and lines in $\pg(2,q)$, $q = p^h$ where $p$ prime and $h \geq 1$, and let $\rho$ be the polarity described above.
		Then the polarity graph $R(D,\rho)$ has a coclique of size at least $c q\sqrt{q} - O(q)$, 
		%not containing any absolute points, 
		where
		\[
		c = \begin{cases}
			1 & \text{for $p=2$ and $h$ even}, \\
			\frac 1 {\sqrt2} & \text{for $p=2$ and $h$ odd}, \\
			\frac 1 2 & \text{for $p$ odd and $h$ even}, \\
			0.19 & \text{for $p$ odd and $h$ odd}.
		\end{cases}
		\]
	\end{thm}
	
	\begin{crl}\label{corprojplane}
		Let $D$ be the design of points and lines in $\pg(2,q)$, $q \geq 101$. Then
		\[q^2-q\sqrt{q}-O(q) \leq \gamma_e(D) \leq q^2-cq\sqrt{q}-O(q),\]
		where $c$ is defined as in \Cref{cocliquepolaritygraph}. 
	\end{crl}
	\begin{proof}
		The lower bound follows from \Cref{CrlDesignDomIneqs}. The upper bound follows from \Cref{LemIncidencFreeFromPolarity} and \Cref{cocliquepolaritygraph}, giving us the incidence-free pair, and \Cref{ThmMainIntro} turning it into an edge-dominating set of the wanted size.
	\end{proof}
	
	Going via the polarity graph is actually a detour when $p = 2$. It is known that $\pg(2,q)$ has a maximal arc of order $n$ if and only if $q = 2^h$ and $n$ divides $q$ \cite{denniston,ball}.
	So starting from a maximal arc
	of order $\sqrt{q}$ or $\sqrt{q/2}$, depending on the parity of $h$, we can combine our earlier observations on arcs
	and the corresponding dual arcs with \Cref{LmBiregular} to improve the result in \Cref{corprojplane} for even $q$. The
	improvement only lies in the omission of the restriction on $q$, the actual value of $\gamma_e(D)$ will be the same.
	We state the sharp result for even $h$ for completeness.
	\begin{crl}
		For all $q = 2^{2h}$ we have 
		\[\gamma_e(D) = q^2-q\sqrt{q}+\sqrt{q}+1.\]
	\end{crl}
	
	A similar result holds in the more general design of points and hyperplanes in $\pg(n,q)$, $n \geq 2$. A large incidence-free set was found by De Winter, Schillewaert, and Verstra\"ete \cite[\S 4.2]{dewinterschillewaertverstraete}.
	
	\begin{thm}
		\label{ResIncidenceFreeHyperplanes}
		There exists an equinumerous incidence-free pair $(X,Y)$ of sets of points and hyperplanes respectively in $\pg(n,q)$ of size $\Theta (q^{\frac{n+1}2})$.
	\end{thm}
	
	\begin{crl}
		Let $D$ be the design of points and hyperplanes in $\pg(n,q)$, $n \geq 2$.
		Then
		\[
		\gamma_e(D) = \theta_n - \Theta_q (q^{\frac{n+1}2}).
		\]
	\end{crl}
	
	\begin{proof}
		This follows from \Cref{CrlDesignDomIneqs} and the preceding result combined with \Cref{ThmMainIntro}.
		%The bound from Theorem \ref{CrlDesignDomIneqs}(2) yields
		%\[
		% \gamma_e(D) \geq \theta_n - \left(1 + \frac{q^n-1}{q^{n-1}-1}(q^{\frac{n-1}{2}} - 1) \right)
		% = \theta_n - \Theta_q (q^{\frac{n+1}2}).
		%\]
		%Now we prove that there exists an edge dominating set in $I(D)$ of size $\theta_n - \Theta_q (q^{\frac{n+1}2})$.
		%For $n=2$ this was done in the previous part of this section.
		%For $n \geq 3$, $\lambda = \theta_{n-2}$ and $r = \theta_{n-1}$, which implies that the condition $\lambda(\lambda+1) \geq r$ of Lemma \ref{LmHallPair} is met.
		%Combining this with Result \ref{ResIncidenceFreeHyperplanes} and Lemma \ref{LmHallPairToMatching} yields the desired edge dominating set.
	\end{proof}
	
	In other projective point-subspace designs, it is not difficult to determine the edge domination number exactly.

	\begin{prop}
		Let $D$ be the design of points and $k$-spaces in $\pg(n,q)$, $1 \leq k < n-1$, and $n \geq 3$.
		Then 
		$$\gamma_e(D) = 
		\begin{cases} 
			\theta_n - q & \text{if } k=1, \\
			\theta_n & \text{otherwise.}
		\end{cases}$$
	\end{prop}
	
	\begin{proof}
		If $1 < k < n-1$, then the number of $k$-spaces each point is incident to equals
		\[
		r = \prod_{i=1}^{k} \frac{q^{n-k+i}-1}{q^i-1} > \frac{q^{n+1}-1}{q-1} = v.
		\] 
		By Lemma \ref{LmLargeRepNumber}, $\gamma_e(D) = v = \theta_n$.
		
		Now suppose $k=1$.
		First we show that $\gamma \geq \theta_n - q$.
		%This can be done using Theorem \ref{CrlDesignDomIneqs}(3) or by checking that
		One can see that
		\[
		\theta_n - (q+1) < (q+1) \theta_{n-1} - \binom{q+1}{2}
		\]
		if $n \geq 3$.
		Therefore, by Lemma \ref{LmIneqCap} every set $S$ of $q+1$ points is incident with a set of lines $T$, where $|T| > \theta_n - (q+1)$. Now suppose for the sake of contradiction that we have an edge dominating set $\Gamma$ with $|\Gamma| \leq \theta_n-(q+1)$. Then there are at least $q+1$ points not covered by an edge of $\Gamma$, so that the lines with which they are incident should all be covered by an edge. However, the number of such lines is more than $\theta_n - (q+1) = |\Gamma|$, so this is impossible.
		
		We can describe a construction that yields edge dominating sets of $I(D)$ of size $\theta_n - q$.
		Consider an incident point-line pair $(x,\ell)$. 
		Let $S$ denote the points not on $\ell$ and $T$ denote the set of lines intersecting $\ell \setminus \{x\}$ in a point (so not $\ell$ itself). 
		Then $|S| = |T| = \theta_n-(q+1)$. The subgraph of $I(D)$ induced on $(S,T)$ is $q$-regular and hence has a perfect matching due to \Cref{LmBiregular}. By adding the edge $(x,\ell)$, we find a maximal matching in $I(D)$ of size $\theta_n-q$.
		%Take any set $S$ of $q$ points in $\pg(n,q)$.
		%Let $\mp$ denote the points not in $S$ and $\ml$ denote the set of lines intersecting $S$.
		%
		%Consider the incidence graph $H$ of $(\mp,\ml)$.
		%Then $H$ satisfies condition (H3) of Result \ref{ResHall}.
		%Indeed, take a line $l \in \ml$.
		%Then $l$ contains $x$ points of $S$, with $1 \leq x \leq q$.
		%Therefore, $l$ has degree $q + 1 - x > 0$ in $H$.
		%Now take a point $P \in l \setminus S$.
		%Then there are $q-x$ points of $S$ not on $l$.
		%Hence, $P$ lies on at most $q-x$ other lines intersecting $S$.
		%This implies that the degree of $P$ is at most $q+1-x$.
		%
		%Thus, $H$ has a matching $M$ covering $\ml$.
		%Since $H$ is an induced subgraph of $I(D)$, this is also a matching in $I(D)$.
		%For every point $P$ not in $S$ and not covered by $M$, choose an arbitrary line $l$ through $P$, and add $(P,l)$ to $M$.
		%Then $M$ is an edge dominating set of $I(D)$ of size $\theta_n - q$.
	\end{proof}

	\subsection{Symmetric designs from Hadamard matrices}
	A \emph{Hadamard matrix} is an $n \times n$-matrix $M$ with only $1$ and $-1$ as entries such that $M^t M = n I_n$.
	We refer the reader to \cite[Chapter 4]{ioninshrikhande} for a treatise of the subject.
	There are several constructions of symmetric designs from Hadamard matrices.
	
	\subsubsection{Hadamard designs and Paley matrices}
	
	Let $\one$ denote the all-one column vector.
	If a $4u \times 4u$ Hadamard matrix $M$ is normalised, i.e.\ of the form
	\[
	M = \begin{pmatrix} 1 & \one^t \\ \one & N \end{pmatrix}, 
	\]
	then we can replace every $-1$ in $N$ with a 0. Equivalently, if we define $J$ to be the all-one matrix of the appropriate size (which will be clear from context), then we consider the matrix $\half(N+J)$.
	This yields the incidence matrix of a symmetric $2-(4u-1,2u-1,u-1)$ design.
	Such a design is called a \emph{Hadamard design}.
	The next bound follows directly from 
	%Lemma \ref{LemBlocksThroughPoints}(2). 
	\Cref{CrlDesignDomIneqs}(2).
	\begin{lm}
		Let $D$ be a $(4u-1,2u-1,u-1)$ Hadamard design.
		Then
		\[
		\gamma_e(D) \geq 4u - 2 \sqrt u - \frac 1 {\sqrt u + 1}.
		\]
	\end{lm}
	
	We consider the following construction due to Paley \cite{paley} of a Hadamard matrix.
	Consider the field $\FF_q$ with $q$ odd.
	Define the quadratic character of $\FF_q$ as
	\[
	\chi(x) = \begin{cases}
		0 & \text{if } x = 0, \\
		1 & \text{if $x$ is a non-zero square}, \\
		-1 & \text{if $x$ is not a square},
	\end{cases}
	\]
	Let $Q$ denote the Jacobstahl matrix of $\FF_q$, i.e.\ the rows and columns of $Q$ are indexed by the elements of $\FF_q$, and $Q_{x,y} = \chi(x-y)$.
	
	Suppose that $q \equiv 1 \pmod 4$.
	Then the matrix
	\[
	M = \begin{pmatrix}
		0 & \one^t \\
		\one & Q
	\end{pmatrix} \otimes
	\begin{pmatrix}
		1 & 1 \\ 1& -1
	\end{pmatrix} +
	I \otimes \begin{pmatrix}
		1 & -1 \\ -1 & -1
	\end{pmatrix}
	\]
	is a Hadamard matrix, where $\otimes$ denotes the tensor product.
	We can interpret this matrix as follows.
	Say that the the first row and column of $\begin{pmatrix}
		0 & \one^t \\
		\one & Q
	\end{pmatrix}$
	are indexed by $\infty$.
	The other rows and columns are naturally indexed by $\FF_q$.
	Say that the rows and columns of the $2 \times 2$-matrices are indexed by $1$ and $-1$, in that order.
	Define $\infty - \infty = 0$, $\infty - x = x - \infty = \infty$ for all $x \in \FF_q$, and $\chi(\infty) = 1$.
	Then
	\[
	M((x,i),(y,j)) = \begin{cases}
		1 & \text{if $x=y$ and } (i,j)=(1,1), \\
		-1 & \text{if $x=y$ and } (i,j) \neq (1,1), \\
		\chi(x-y) & \text{if $x\neq y$ and } (i,j) \neq (-1,-1), \\
		- \chi(x-y) & \text{if $x \neq y$ and } (i,j)=(-1,-1).
	\end{cases}
	\]
	If we scale the second the second row and the second column with a factor $-1$, then we obtain a Hadamard matrix of the form $\begin{pmatrix}
		1 & \one^t \\ \one & N
	\end{pmatrix}$, and therefore $\frac 1 2 (N+J)$ is the incidence matrix of a Hadamard design.
	We denote this design as $HD(q)$.
	It is a $(2q+1,q,\frac{q-1}2)$ design.
	The point set is $(\FF_q \times \set{1,-1}) \cup \set \infty$.
	For each point $P$ there is a block $B_P$, with 
	\begin{align*}
		B_\infty &= \FF_q \times \set {-1}, \\
		B_{(x,1)} &= \sett{(y,i)}{\chi(x-y)=1, i=\pm 1} \cup \set{(x,1)}, \\
		B_{(x,-1)} &= \sett{(y,i)}{\chi(y-x)=i} \cup \set \infty.
	\end{align*}
	
	\begin{lm}
		Let $C$ be a clique in the Paley graph of order $q \geq 101$.
		Then $I(HD(q))$ has an edge dominating set of size $2q+1 - |C|$.
	\end{lm}
	
	\begin{proof}
		Consider the map $\rho$ which maps a point $P$ to the block $B_P$, and vice versa the block $B_P$ to the point $P$.
		As $N$ is symmetric, for any two points $P$ and $Q$, it holds that $P \in Q^\rho$ if and only if $Q \in P^\rho$.
		In other words, $\rho$ is a polarity of the design.
		The cocliques in the polarity graph of $HD(q)$ are exactly the sets $K \times \set {-1}$, with $K$ a clique in the Paley graph of order $q$, and $\set \infty$.
		
		Hence, if $C$ is a clique in the polarity graph, then the statement of the lemma follows from \Cref{ThmMainIntro} and \Cref{LemIncidencFreeFromPolarity}.
	\end{proof}
	
	If $q$ is a square, then it is known that $\FF_{\sqrt{q}}$ is a clique of maximal size in the Paley graph of order $q$. We thus obtain the following corollary. 
	
	\begin{crl}
		If $q \geq 101$ is a square, then $\gamma_e(HD(q)) = 2q - \Theta (\sqrt{q})$.
	\end{crl}

	\subsubsection{Menon designs and symmetric Bush-type Hadamard matrices}
	
	If a $4u \times 4u$ Hadamard matrix $M$ has constant row sum, then it is called regular.
	In that case $u$ must be square, say $u =h^2$, and $\frac 1 2 (J+M)$ and $\frac 1 2 (J-M)$ are incidence matrices of symmetric designs, called \emph{Menon designs}.
	These designs have parameters $2$-$(4h^2,2h^2+\eps h,h^2 + \eps h)$ with $\eps = \pm 1$.
	By Lemma \ref{LemBlocksThroughPoints}(2), the edge domination number of such a design is at least $4h^2 - 2h$ if $\eps = -1$, and $4h^2 - 2h + 2 - \frac{2}{h+1}$ if $\eps = 1$.
	
	A regular $4h^2\times 4h^2$ Hadamard matrix $M$ is \emph{of Bush type} if it is a block matrix where all the blocks are $2h \times 2h$-matrices, all the diagonal blocks are all-one matrices, and all the other blocks have constant row and column sum 0.
	If such a matrix $M$ is symmetric, then $\frac 1 2 (J-M)$ is the incidence matrix of Menon design with $\eps = -1$, which has a polarity without absolute points, and where the vertices of the polarity graph can be partitioned into cocliques of size $2h$.
	Note that each such coclique consists of $2h$ points, with $2h$ blocks not incident with these points (namely their image under the polarity).
	
	\begin{prop}
		If $D$ is a Menon design associated with a symmetric Hadamard matrix of Bush type, it is a $(4h^2,2h^2-h,h^2 - h)$-design and $\gamma_e(D) = 4h^2 - 2h$.
	\end{prop}
	
	\begin{proof}
		Applying \Cref{CrlDesignDomIneqs}(2) to this set of parameters gives the lower bound, while the incidence-free set of size $2h$ described above provides the upper bound. We note that this set is a maximal arc by \Cref{PropUpperBoundIF}(2) and so indeed gives rise to an edge dominating set as seen before.
	\end{proof}
	
	There are constructions of infinite families of symmetric Bush-type Hadamard matrices. One such family can be found in \cite{muzychuckxiang}.
	Another interesting family are the symplectic symmetric designs, which Kantor \cite{kantor1985} proved to be one of the only two infinite families of 2-transitive symmetric designs with $v>2k$.

	\subsection{Semi-biplanes}\label{SubsecSemibiplanes}
	
	There are only a very limited number of explicit symmetric designs with $\lambda=2$.
	These designs are known as \emph{biplanes}.
	However, we can relax the definition to obtain the so-called semi-biplanes, of which there are several known infinite families.
	They were introduced by Hughes \cite{hughes}.
	
	\begin{df}
		A \emph{semi-biplane} is a symmetric incidence structure $D = (\mp,\mb)$ such that
		\begin{enumerate}
			\item for every two distinct points $P$ and $Q$, there are either 0 or 2 blocks incident with $P$ and $Q$,
			\item for every two distinct blocks $B$ and $C$, $|B \cap C|$ equals either 0 or 2,
			\item the incidence graph $I(D)$ is connected.
		\end{enumerate}
	\end{df}
	
	Given a semi-biplane, there exist numbers $v$ and $k$ such that $|\mp|= |\mb| = v$, and every point or block is incident with exactly $k$ blocks or points respectively.
	The parameters $(v,k)$ are called the \emph{order} of the semi-biplane, and a semi-biplane of order $(v,k)$ is denoted as $\sbp(v,k)$.
	
	The \emph{collinearity graph} of an incidence structure $(\mp,\mb)$ is the graph with vertex set $\mp$, where two distinct points $P$ and $Q$ are adjacent if and only if there is some block incident with $P$ and $Q$.
	The collinearity graph of a $\sbp(v,k)$ is regular of degree $\binom k 2$.
	
	If we denote by $N$ the point-block incidence matrix of a $\sbp(v,k)$, and by $A$ the adjacency matrix of its collinearity graph, then
	\[
	N N^t = k I_v + 2 A.
	\]
	
	\subsubsection{Divisible semi-biplanes}
	
	We call a semi-biplane $D = (\mp,\mb)$ \emph{divisible} if its collinearity graph is a complete multipartite graph.
	Another way to formulate is that the relation $\sim$ defined on $\mp$ by
	\[
	P \sim Q \iff P=Q \text{ or $P$ and $Q$ are not collinear}
	\]
	is an equivalence relation.
	It is not difficult to show that every equivalence class contains the same number of elements, say $d$, and that $v = \binom k 2 + d$.
	
	\bigskip
	
	We would like to give a lower bound on the edge domination number of the incidence graph of a divisible semi-biplane.
	As before, such a bound follows from an upper bound on the size of an equinumerous incidence-free pair.
	One way to bound this size is by using \Cref{LemBlocksThroughPoints}.
	Another way is by using an eigenvalue bound from spectral graph theory.
	For designs, this yields the same bound, but for divisible semi-biplanes, this is no longer the case.
	Both approaches yield lower bounds on $\gamma_e(D)$ that look roughly like $v - \frac{k \sqrt k}2$, but the spectral approach yields a significantly simpler expression.

	\begin{lm}[{\cite[Theorem 5.1]{haemers}} Expander mixing lemma]
		\label{LmExpanderMixing}
		Let $G$ be a $k$-regular bipartite graph, with bipartition $L$ and $R$.
		Take sets $S \subseteq L$ and $T \subseteq R$.
		%Let $e(S,T)$ denote the number of edges with an endpoint in $L$ and an endpoint in $R$.
		Suppose that the second largest eigenvalue of the adjacency matrix of $G$ is $\lambda_2$.
		Then
		\[
		\left( e(S,T) - \frac{k}{|R|} |S| |T| \right)^2 \leq \lambda_2^2 |S| |T| \left( 1 - \frac{|S|}{|L|} \right)\left( 1 - \frac{|T|}{|R|} \right).
		\]
	\end{lm}
	
	\begin{prop}
		\label{PropEML}
		Let $D$ denote an SIS of type $(v,k,\lambda)$.
		Let $\lambda_2$ denote the second largest eigenvalue of $I(D)$.
		Then
		\[
		\gamma_e(D) \geq \frac k {\lambda_2 + k} v.
		\]
	\end{prop}
	
	\begin{proof}
		Take an edge dominating set of size $\gamma := \gamma_e(D)$.
		Then there exist sets $S$ and $T$ of blocks and points respectively of size $v - \gamma$ with no incidences between them.
		Apply the expander mixing lemma in $I(D)$ to $S$ and $T$.
		This tells us that
		\[
		\left( \frac k v (v-\gamma)^2 \right)^2 \leq \lambda_2^2 (v-\gamma)^2 \left( \frac \gamma v \right)^2.
		\]
		Taking square roots, this reduces to
		\[
		k (v-\gamma) \leq \lambda_2 \gamma,
		\]
		which gives the desired equality.
	\end{proof}
	
	\begin{rmk}
		Observe that if $N$ is the incidence matrix of an SIS of type $(v,k,\lambda)$, then $vk = \mathrm{Tr}(NN^t) \leq k^2 + (v-1)\lambda_2$. This shows that $\lambda_2 = \Omega(\sqrt{k})$. On the other hand, we can have $\lambda_2 = k$ when the SIS is the union of two disjoint symmetric designs. This shows that the edge domination number of an SIS could lie anywhere between $\left(1-\Omega \left(\frac{1}{\sqrt{k}} \right) \right)v$ and $\half v$. For the symmetric designs we have seen before, we observe that the truth is closer to the higher end. This is again related to the fact that incidence graphs of symmetric designs are `expanding', this time expressed by the fact that $\lambda_2$ is small compared to $k$.
	\end{rmk}

	\begin{lm}
		\label{LmSpectrumSbp}
		Let $D$ be a divisible $\sbp(v,k)$ and write $d = v - \binom k 2$.
		Then the spectrum of the adjacency matrix of $I(D)$ equals
		$$\{\pm k, \pm \sqrt{k}^{(\frac v d (d-1))}, \pm \sqrt{k-2d}^{(\frac vd -1)}\}.$$
	\end{lm}
	
	\begin{proof}
		Let $N$ denote the incidence matrix of $D$.
		It suffices to compute the spectrum of $N N^t$.
		The adjacency matrix of the collinearity graph of $D$ equals $A = (J-I)_{\frac v d} \otimes J_d$.
		Since $(J-I)_{\frac v d}$ has spectrum $(\frac v d - 1)^{(1)}, -1^{(\frac v d -1)}$, and $J_d$ has spectrum $d^{(1)}, 0^{(d-1)}$, the spectrum of $N N^t = k I + 2 A$ equals \[
		\{{k^2}, k^{(\frac v d (d-1))}, (k-2d)^{(\frac v d - 1)}\}.\]
		
		It is well-known that if $N$ is a square matrix, then $\lambda$ is an eigenvalue of $NN^t$ with multiplicity $m$ if and only if $\sqrt \lambda$ and $- \sqrt \lambda$ are eigenvalues of $\begin{pmatrix}
			O & N \\ N^t & O
		\end{pmatrix}$, both with multiplicity $m$.
	\end{proof}
	
	\begin{rmk}
		The lemma implies that $d \leq k/2$ and thus $v \leq k^2/2$.
		This also has an easy combinatorial proof, see Wild \cite[Result 2]{wild81}.
	\end{rmk}
	
	\begin{crl}\label{CrlDivisibleSBP}
		Let $D$ be a divisible $\sbp(v,k)$. Then
		\[
		\gamma_e(D) \geq v - \frac{v}{\sqrt k + 1}
		\]
	\end{crl}
	
	\begin{proof}
		This follows directly from \Cref{LmSpectrumSbp} and \Cref{PropEML}.
	\end{proof}
	
	\subsubsection{Divisible semi-biplanes from projective planes}
	
	A classical construction of a divisible $\sbp(v,k)$ is given by Hughes, Leonard, and Wilson \cite{hughes}.
	Take an involution $\varphi$ of $\pg(2,q)$, i.e.\ an incidence-preserving map of order two.
	Then either $\varphi$ is a \emph{perspectivity}, which means it fixes a point, called the \emph{centre} of $\varphi$, and fixes a line pointwise, called the \emph{axis} of $\varphi$; or $\varphi$ is a \emph{Baer involution}, i.e.\ it fixes a Baer subplane.
	Take as point set the points $P$ of $\pg(2,q)$ not fixed by $\varphi$, where we consider $P$ and $P^\varphi$ as equivalent, i.e.\ $\mp = \sett{ \set{P,P^\varphi}}{P \neq P^\varphi}$.
	For each line $l$ such that $l^\varphi \neq l$ define the block $B_l = \sett{\set{P,P^\varphi}}{P \in l, \, P \neq P^\varphi}$.
	Then $B_l = B_{l^\varphi}$.
	The biplane is given by $(\mp,\mb)$ with $\mb = \sett{B_l}{l \neq l^\varphi}$.
	An easy counting argument shows that a perspectivity of order 2 in $\pg(2,q)$ is an \emph{elation} (i.e.\ centre and axis are incident) if $q$ is even, and a \emph{homology} (i.e.\ the centre and axis are not incident) if $q$ is odd.
	It is not difficult to verify that two perspectivities of order 2 in $\pg(2,q)$ must be projectively equivalent, and all Baer involutions are projectively equivalent as well.
	This allows us to give a more concrete description of the semi-biplanes described above. For more details on involutions of projective planes, we refer to \cite[page 30]{dembowksi68}.
	
	\bigskip

	{\bf Case 1.} $q$ is even, $\varphi$ is an elation.
	
	Then we can give $\pg(2,q)$ coordinates such that $\varphi:(x_0,x_1,x_2) \mapsto (x_0,x_1,x_2+x_0)$.
	The axis of $\varphi$ is the line $X_0=0$, and the centre is $(0,0,1)$.
	Every point not on the axis has a unique representation $(1,x_1,x_2)$.
	We denote this point as $(x_1,x_2)$.
	Every line not through the centre has a unique equation of the form $X_2 = m X_1 + b$.
	We denote this line as $l_{m,b}$.
	Then $(x_1,x_2)^\varphi = (x_1,x_2+1)$ and $l_{m,b}^\varphi = l_{m,b+1}$.
	If we identify each point and line with its image under $\varphi$, then $(x_1,x_2)$ and $l_{m,b}$ are incident if and only if $x_2 + m x_1 + b \in \set{0,1}$.

	Inspired by Mubayi and Williford \cite{mubayiwilliford}, we give the following construction of a small edge dominating set.
	
	\begin{prop}
		Let $q \geq 128$ be an even prime power, and let $D$ be the $\sbp(q^2/2,q)$ arising from an elation in $\pg(2,q)$.
		Then
		\[
		\frac{q^2}{2}-\frac{q^2}{2(\sqrt{q}+1)} \leq \gamma_e(D) \leq \begin{cases}
			\frac {q^2} 2 - \frac{q \sqrt q}4 & \text{if $q$ is a square}, \\
			\frac {q^2} 2 - \frac{q \sqrt q}{4\sqrt 2} & \text{otherwise.}
		\end{cases}
		\]
	\end{prop}
	
	\begin{proof}
		The lower bound follows from \Cref{CrlDivisibleSBP}, so we focus on the upper bound.
		
		Use the notation for points and lines of this semi-biplane as described above.
		Suppose that $q=2^h$.
		Let $\omega$ be a primitive element of $\FF_q$.
		Then $1, \omega, \dots, \omega^{h-1}$ is an $\FF_2$-basis of $\FF_q$.
		For each $x \in \FF_q$, let $(x^{(0)}, \dots, x^{(h-1)})$ denote its coordinate vector with respect to this basis, i.e.\ $x = \sum_{i=0}^{h-1} x^{(i)} \omega^i$.
		Define $f = \lfloor \frac h 2 \rfloor - 1$, and let $F$ denote $\FF_2$-span of $1, \omega, \dots, \omega^f$.
		If we take $x_1$ and $x_2$ in $F$, then $(x_1 x_2)^{(h-1)} = 0$, since $2 f < h-1$.
		Let $X$ denote the set of all points $(x_1,x_2)$ with $x_1 \in F$ and $x_2^{(h-1)} = 0$.
		Let $Y$ denote the set of all blocks $l_{m,b}$ with $m \in F$ and $b_{h-1} = 1$.
		Then $(x_2 + m x_1 + b)^{(h-1)} = 1$.
		In particular, $x_2 + m x_1 + b \notin \set{0,1}$, which implies that no point of $X$ is incident with a block of $Y$.
		
		Since each point has two coordinate representatives, and likewise for the blocks, we find an equinumerous incidence-free pair of size
		\[
		|X| = |Y| = \frac{|F| \frac q 2}2
		= \frac{2^{f+1} q}4
		= \frac{2^{\lfloor \frac h 2 \rfloor} q}4
		= \begin{cases}
			\frac{q \sqrt q}4 & \text{if $h$ is even,} \\
			\frac{q \sqrt q}{4\sqrt 2} & \text{if $h$ is odd.}
		\end{cases}
		\]
		
		Thus, all that is left to show that $(X,Y)$ satisfy the conditions of \Cref{ThmSISMatching} with $v=q^2/2$, $k = q$ and $\alpha = |X|$ as above.
		
		First off, we have $v = q^2/2$ and $q \geq 128$, which is larger than $9/5(q^2/2)$ and $100$ respectively. Secondly, we have indeed that $|X|\leq q^2/20$ for $q \geq 128$ as $|X| \leq q\sqrt{q}/4$. Finally, take a point $(x_1,x_2)$.
		For every value of $m$, there is a unique $b$ with $(x_1,x_2) \in l_{m,b}$, hence $(x_1,x_2)$ lies on at most $|F| \leq \sqrt q$ lines in $Y$. Equivalently, $(x_1,x_2)$ has at least $q-\sqrt{q} \geq 2|X|/k$ neighbours in $\mb \setminus Y$. 
	\end{proof}
	
	{\bf Case 2.} $q$ is odd, $\varphi$ is a homology.
	
	We can choose coordinates such that $\varphi: (x_0,x_1,x_2) \mapsto (-x_0,x_1,x_2)$.
	The axis of $\varphi$ is the line $X_0=0$, the centre of $\varphi$ is $(1,0,0)$.
	Similar to the previous case, we represent each point of the semi-biplane as $(x_1,x_2) \neq (0,0)$, where $(x_1,x_2)$ and $(-x_1,-x_2)$ are considered to be the same point.
	Every line in $\pg(2,q)$ distinct from $X_0=0$ that misses $(1,0,0)$ has a unique equation of the form $X_0 = a X_1 + b X_2$, $(a,b) \neq (0,0)$.
	Denote this line as $l_{a,b}$.
	Then $l_{a,b}^\varphi = l_{-a,-b}$.
	The lines in the biplane are of the form $l_{a,b}$ with $(x_1,x_2) \in l_{a,b}$ if and only if $a x_1 + b x_2 = \pm 1$.

	\begin{prop}
		Let $q \geq 11$ be odd, and let $D$ be the $\sbp((q^4-1)/2,q^2)$ arising from a homology in $\pg(2,q^2)$.
		Then
		\[
		\frac{q^4-1}{2}-\frac{q^4-1}{2q+1}\leq \gamma_e(D) \leq \frac{q^4-1}2 - q\frac{q^2-1}4.
		\]
	\end{prop}
	
	\begin{proof}
		The lower bound again follows from \Cref{CrlDivisibleSBP}.
		
		Let $\zeta$ denote a primitive $(q+1)$st root of unity in $\FF_{q^2}$.
		Consider the sets
		\begin{align*}
			X &= \sett{(x_1,x_2)}{ x_1 \in \FF_q, \, x_2^{q-1} \in \set{\zeta^0, \dots, \zeta^{(q-1)/2} } } \\
			Y &= \sett{l_{a,b}}{ a \in \FF_q, \, b^{q-1} \in \set{\zeta^1, \dots, \zeta^{(q+1)/2} }}
		\end{align*}
		If you take a point $(x_1,x_2) \in X$ and a line $l_{a,b} \in Y$, then $a x_1 \in \FF_q$ and $(b x_2)^{q-1} \in \set{\zeta^1, \dots, \zeta^q}$.
		In particular, $(b x_2)^{q-1} \notin \set{0,1}$, which implies that $b x_2 \notin \FF_q$.
		Therefore, $a x_1 + b x_2 \notin \FF_q$, hence $a x_1 + b x_2 \neq \pm 1$.
		Thus, $(X,Y)$ is an equinumerous incidence-free pair.
		
		To calculate the size of $X$, note that there are $q$ choices for $x_1$, $\frac{q+1}2$ choices for $x_2^{q-1}$, so $\frac{(q-1)(q+1)}2$ choices for $x_2$. Since $(x_1,x_2)$ and $(-x_1,-x_2)$ are the same point and either both or neither are in $X$, we conclude that $|X| = |Y| = \frac{q(q^2-1)}4$.
		
		To finish the proof, we check that the conditions of \Cref{ThmSISMatching} are met.
		Using that $q \geq 11$, the only non-trivial condition to check is that every point outside of $X$ is incident with at least $\frac 2 {q^2} q \frac{q^2-1}4 = \frac{q^2-1}{2q}$ lines outside of $Y$ and vice versa.
		By the symmetry of the situation, we only check the condition for points outside of $X$.
		So take a point $(x_1,x_2) \notin X$. 
		
		First suppose that $x_2 \neq 0$.
		Then for every $a \notin \FF_q$, the lines $l_{a,\frac{\pm 1 - ax_1}{x_2}}$ are lines outside of $Y$ incident with $(x_1,x_2)$.
		So $(x_1,x_2)$ is incident with at least $q^2-q > \frac{q^2-1}{2q}$ lines outside of $Y$.
		
		Now suppose that $x_2 = 0$.
		Then $x_1 \neq 0$, and every line $l_{\frac{\pm 1}{x_1},b}$ is incident with $(x_1,x_2)$.
		There are $q^2 - \frac{q^2-1}2 = \frac{q^2+1}2$ values of $b$ such that $b^{q-1} \notin \set{\zeta^1, \dots, \zeta^{(q+1)/2}}$, so $(x_1,x_2)$ is incident with at least $\frac{q^2+1}2 \geq \frac{q^2-1}{2q}$ lines outside of $Y$.
	\end{proof}
	
	\begin{rmk}
		The semi-biplanes with $\varphi$ a homology and $q$ an odd non-square prime power are left untreated.
		If $\FF_q$ has a large subfield, a similar construction yields a fairly large incidence-free pair.
		However, if $\FF_q$ does not have a large subfield, one would need some new ideas.
	\end{rmk}
	
	{\bf Case 3.} $\varphi$ is a Baer involution in $\pg(2,q^2)$.
	
	After choosing appropriate coordinates, $\varphi: (x_0,x_1,x_2) \mapsto (x_0^q,x_1^q,x_2^q)$.
	The semi-biplane then consists of the points $(x_0,x_1,x_2)$ not fixed by $\varphi$, which we identify with $(x_0^q,x_1^q,x_2^q)$, and the blocks $l_{a,b,c}$ with $(a,b,c) \neq (a^q,b^q,c^q)$ containing the points $(x_0,x_1,x_2)$ satisfying $a x_0 + b x_1 + c x_2 = 0$ or $a x_1^q + b x_1^q + c x_2^q = 0$.
	Note that $l_{a,b,c} = l_{a^q,b^q,c^q}$.
	
	\begin{prop}
		Let $q\geq 11$ be a prime power, and let $D$ denote the $\sbp(q(q^3-1)/2,q^2)$ arising from a Baer involution in $\pg(2,q^2)$.
		Then
		\[
		\frac{q^4-q}{2}-\frac{q^4-q}{2q+1} \leq \gamma_e(D) \leq
		\begin{cases}
			q\frac{q^3-1}2 - \frac{q-1}2 \frac{q^2+2q-1}2 & \text{if } q \equiv 3 \pmod 4, \\
			q\frac{q^3-1}2 - (q^2-1) \floor{ \frac q 4 } & \text{otherwise.}
		\end{cases}
		\]
	\end{prop}
	
	\begin{proof}
		The lower bound follows from \Cref{CrlDivisibleSBP}.
		
		Let $\zeta$ be a primitive $(q+1)$st root of unity in $\FF_{q^2}$.
		Define $f = \floor{\frac{q+1}4}$.
		Consider the sets
		\begin{align*}
			X & = \sett{ (x_0,x_1,x_2) }{ (x_0, x_1) \in \pg(1,q), \, x_2^{q-1} \in \set{\zeta, \dots, \zeta^f}  }, \\
			Y & = \sett{ l_{a,b,c} }{ (a, b) \in \pg(1,q), \, c^{q-1} \in \set{\zeta^{f+1}, \dots, \zeta^{2f}} }.
		\end{align*}
		
		Take $(x_0,x_1,x_2) \in X$ and $l_{a,b,c} \in Y$.
		Then $a x_0 + b x_1 = a x_0^q + b x_1^q \in \FF_q$, $x_2^{q-1} = \zeta^i$ with $1 \leq i \leq f$, and $c^{q-1} = \zeta^j$ with $f+1 \leq j \leq 2f$.
		Therefore, $(c x_2)^{q-1} = \zeta^{i+j}$ with $f+1 \leq i+j \leq 3f$.
		On the other hand, since $\zeta^q = \zeta^{-1}$, $(x_2^q c)^{q-1} = \zeta^{j-i}$ with $1 \leq j-i \leq 2f -1$.
		Thus, $(c x_2)^{q-1}$ and $(c x_2^q)^{q-1}$ are not equal to 1.
		This means that $c x_2$ and $c x_2^q$ are not in $\FF_q$, which implies that $a x_0 + b x_1 + c x_2$ and $a x_0^q + b x_1^q + c x_2^q$ are not in $\FF_q$, so definitely not equal to 0.
		Hence, no point of $X$ lies on a block of $Y$.
		
		For every $(q+1)$st root of unity $\zeta^i$, there are $q-1$ elements $x_2$ with $x_2^{q-1} = \zeta^i$, which in fact form a multiplicative coset of $\FF_q^*$.
		Hence, $|X| = (q+1)(q-1)f$.
		Note that we do not count any points twice under the equivalence $(x_0,x_1,x_2) = (x_0^q,x_1^q,x_2^q)$, since $(x_0,x_1) \in \pg(1,q)$ implies that $(x_0^q,x_1^q) = (x_0,x_1)$, and $x_2^{q-1} \in \set{\zeta, \dots, \zeta^f}$ implies that $(x_2^q)^{q-1} \in \set{\zeta^q, \dots, \zeta^{q+1-f}}$.
		
		For the set $Y$, things are a little more delicate.
		If $f = \frac{q+1}4$, i.e.\ $q \equiv 3 \pmod 4$, we must beware that if $c^{q-1} = \zeta^{2f}$, then $l_{a,b,c}$ and $l_{a,b,c^q}$ define the same line, but $(c^q)^{q-1}$ also equals $\zeta^{2f}$.
		Thus, in this case, the size of $Y$ equals $(q+1)((f-1)(q-1) + \frac{q-1}2) = (q^2-1)\frac{q-1}4$.
		Consider the blocks $l_{a,1,0}$ with $a \notin \FF_q$.
		If $(x_0,x_1,x_2) \in X$, then $a x_0 + x_1$ cannot be zero, since $a x_0 \notin \FF_q$ and $x_1 \in \FF_q$.
		There are $\frac{q^2-q}2$ such blocks.
		Define $Y' = Y \cup \sett{l_{a,1,0}}{a \notin \FF_q }$.
		Then $|Y'| = |X| - \frac{q-1}2$.
		So we can delete $\frac{q-1}2$ points from $X$ to obtain a subset $X'$.
		Then $(X', Y')$ is an equinumerous incidence-free pair.
		
		To finish the proof, one can again check that the conditions of \Cref{ThmSISMatching} are satisfied. Apart from some elementary inequalities, one needs to observe that every point $(x_0,x_1,x_2)$ in $\mp \setminus X$ is contained in at most $q+1$ blocks of $Y$. First suppose that $x_2=0$ so that we can write it as $(x_0,1,0)$ with $x_0 \notin \FF_q$. Then it cannot be incident with a block in $Y$ as otherwise $ax_0+b=0$ or $ax_0^q+b=0$ leads to the contradiction that $x_0 \in \FF_q$. Now suppose that $x_2 \neq 0$. Choose $(a,b) \in \pg(1,q)$, then $ax_0+bx_1+cx_2 = 0$ has a unique solution for $c$. Hence for this choice of $(a,b)$, there is at most one $l_{a,b,c} \in Y$ incident with the point $(x_0,x_1,x_2)$.
	\end{proof}

	\subsubsection{Semi-biplanes from binary affine spaces}
	
	In this subsection, we consider another large family of semi-biplanes. They are no longer divisible, but we can still find lower bounds on the edge-domination number, emphasising the flexibility of the approach based on eigenvalues. In one particular member of this family of semi-biplanes, we will construct a large equinumerous incidence-free pair in an SIS that will not give rise to an edge dominating set.
	This shows the necessity of some extra conditions in \Cref{ThmSISMatching} when compared with \Cref{PropHallPair}. \\
	
	Consider the $n$-dimensional affine space over $\FF_2$, denoted $\ag(n,2)$.
	Give coordinates to the points.
	The \emph{weight} of a point is the number of coordinate positions in which the point has a non-zero entry.
	Let $W$ and $\mp$ denote the sets of points of odd and even weight, respectively.
	Consider a set $S \subseteq W$ of size $k$ such that
	\begin{itemize}
		\item the (affine) span of $S$ is $W$,
		\item $S$ does not fully contain any (affine) plane of $\ag(n,2)$.
	\end{itemize}
	Define $\mb = \sett{y + S}{y \in W}$.
	Then $D = (\mp,\mb)$ is an $\sbp(2^{n-1},k)$, see \cite{wild95}.
	
	\begin{lm}
		Let $\mathcal H$ denote the set of all affine hyperplanes of $W$.
		Then the spectrum of the adjacency matrix of $I(D)$ equals the multiset
		\[
		\set{\pm k} \cup \sett{k-2|H \cap S|}{H \in \mathcal H}.
		\]
	\end{lm}
	
	\begin{proof}
		As in \Cref{LmSpectrumSbp}, let $N$ denote the incidence matrix of the semi-biplane, and $A$ the adjacency matrix of its collinearity graph.
		Then $N N^t = k I + 2 A$, and the spectrum of the adjacency matrix of $I(D)$ can be derived from the spectrum of $N N^t$.
		
		Denote $S^+ = \sett{ s+t }{ s,t \in S, \, s \neq t }$.
		The planes in $\ag(n,2)$ are exactly the sets of four distinct points whose total sum is the zero vector.
		Since $S$ does not contain any plane, $s+t$ uniquely determines $\set{s,t} \subseteq S$, and $|S^+| = \binom k 2$.
		Furthermore, two distinct points $x,z \in \mp$ are collinear if and only if $x, z \in y + S$ for some $y \in W$.
		This is equivalent to $x = y + s$ and $z = y + t$ for some $y \in W$, and some $s,t \in S$.
		This again is equivalent to $x + s = z + t$ for some $s,t \in S$, since this immediately implies that $x+s=z+t \in W$.
		The last statement is equivalent to $x + z \in S^+$.
		In conclusion, $x$ and $z$ are collinear if and only if $x + z \in S^+$.
		
		We can now use \cite[\S 7.1]{brouwervanmaldeghem} to conclude that the spectrum of $A$ equals the multiset
		\[
		\set{ \binom k 2 } \cup \sett{ 2 |H \cap S^+| - |S^+| }{ H \text{ a hyperplane of $\mp$ through } \zero}.
		\]
		
		Consider the standard inner product $x \cdot y = \sum_i x_i y_i$ on $\ag(n,2)$.
		Every hyperplane $H$ in $\mp$ through $\zero$ is of the form $X \cdot a = 0$ for some $a \in \mp \setminus \set {\zero,\one}$.
		Then $H$ is an $(n-2)$-space in $\ag(n,2)$.
		There are three hyperplanes in $\ag(n,2)$ through $H$, namely $H_0$ with equation $X \cdot a = 0$, $H_1$ with equation $X \cdot (a + \one) = 0$, and $\mp$.
		Then $H_0 \cap W$ and $H_1 \cap W$ are parallel hyperplanes of $W$, partitioning the points.
		Take a point $x$ in $S^+$.
		There exist unique $s$ and $t$ in $S$ such that $x = s+t$.
		Then 
		\[
		x \notin H \iff x \cdot a = 1 \iff s \cdot a \neq t \cdot a.
		\]
		Thus, $|S^+ \setminus H|$ equals the number of ways to choose an element of $s \in S \cap H_0$ and an element $t \in S \cap H_1$.
		If $|S \cap H_0| = m$, then $|S \cap H_1| = k-m$.
		Hence, 
		\[ 
		2 |H \cap S^+ | - |S^+| = |S^+| - 2 |H \setminus S^+| = \binom k 2 - 2 m (k-m) 
		\]
		This gives an eigenvalue
		\[
		k + 2 \left( \binom k 2 - 2m(k-m) \right)
		= k + k(k-1) - 4m(k-m) = k^2 - 4m(k-m)
		= (k-2m)^2
		\]
		of $N N^t$.
		This gives us eigenvalues $k-2m$ and $-(k-2m)$ of the adjacency matrix of $I(D)$.
		Note that if $k - 2 |S \cap H_0| = k-2m$, then $k - 2 |S \cap H_1| = -(k-2m)$.
		Lastly, the eigenvalue $\binom k 2$ of $A$, gives us the eigenvalue $k^2$ of $N N^t$, hence the eigenvalues $\pm k$ of the adjacency matrix of $I(D)$.
		This proves the statement of the lemma.
	\end{proof}
	
	The following lower bound on the edge domination number of $I(D)$ follows directly from \Cref{PropEML}.
	
	\begin{lm}
		Let $W$ denote the hyperplane of odd-weight points in $\ag(n,2)$, and let $S$ be a subset spanning $W$, not containing a plane.
		Denote the associated $\sbp(2^{n-1},k)$ by $D$.
		Let $m$ denote the maximum number of points of $S$ contained in a hyperplane of $W$.
		Then
		\[
		\gamma_e (D) \geq \frac k m 2^{n-2}.
		\]
	\end{lm}
	
	\begin{rmk}\label{RmkConnectedCounterexample}
		Let $S$ be the set of weight one vectors. In this $\sbp(2^{n-1},n)$, we can construct a large incidence-free set by taking $X$ to be the set of even weight vectors with weight at most $\floor{n/2}-1$ while we take $Y$ to be the set of blocks $y + S$, where $y$ runs over the odd weight vectors with weight at least $\ceil{n/2}+1$.
		By taking subsets, we can easily find a large equinumerous incidence-free pair of sets, but it will not give rise to a small edge dominating set in its complement. This is easily seen, as a point corresponding to a vector of even weight at least $\ceil{n/2}+2$ has no neighbours in $\mb \setminus Y$. This example reaffirms that an assumption like the minimum degree condition in \Cref{ThmSISMatching} is necessary, even though $I(D)$ is connected.
	\end{rmk}
	
	\section{Conclusion}\label{SecConclusion}
	
	In this paper, we studied the edge domination number $\gamma_e(D)$ of incidence structures $D$ through its connections with maximal matchings and incidence-free sets.
	In almost all families we studied, we saw that $\gamma_e(D) = (1-o(1)) v$, so the interesting quantity to study is $v - \gamma_e(D)$.
	Using a combination of probabilistic, combinatorial and geometric techniques, supplemented with tools from spectral graph theory, we made headway on bounding the edge domination number of various designs, often obtaining sharp bounds on $v - \gamma_e(D)$ up to a constant factor. Nevertheless, many problems remain open, and we state a few of them here. 
	
	%In various classes of symmetric designs and incidence structures, we can use the algebraic structure and the symmetries to obtain large incidence-free sets. However, many open cases remain and some lead to interesting combinatorial problems. We mention just a few.
	
	\begin{prob}
		What is the edge domination number for non-Desarguesian projective planes? Can one construct large incidence-free pairs?
	\end{prob}	
	
	Various classes of symmetric designs are constructed using difference sets \cite[Section 6.8]{colbourndinitz}. Quite a few of them are difference sets in cyclic groups.
	
	\begin{prob}
		Is it possible to construct incidence-free pairs in a `uniform' way for symmetric designs coming from difference sets? That is, without resorting to the specific structure of the difference set, but only using properties of the group?
	\end{prob}
	If $G$ is a group with difference set $S$, one can translate this problem to finding so-called cross-intersecting independent sets in the (directed) Cayley graph $C(G,S)$ where we have an arc between two distinct elements $x$ and $y$ if $xy^{-1} \in S$. \\
	
	The following problem asks if we can improve our general lower bound from \Cref{PropProbability}.
	\begin{prob}
		Can we always find equinumerous incidence-free pairs of size $k^{1+\eps}/\lambda$ in a symmetric $(v,k,\lambda)$-design for small $\lambda$?
	\end{prob}
	
	It seems plausible that such incidence-free pairs should always exist for symmetric designs, but we do not necessarily think it should hold for SIS's in general.  Finding an example of such an SIS (if it exists) would also be of significant interest.
	
	\bibliographystyle{alpha}

\begin{thebibliography}{CDJHH04}
		
		\bibitem[BBM97]{ball}
		S.~Ball, A.~Blokhuis, and F.~Mazzocca.
		\newblock Maximal arcs in {D}esarguesian planes of odd order do not exist.
		\newblock {\em Combinatorica}, 17(1):31--41, 1997.
		
		\bibitem[BVM22]{brouwervanmaldeghem}
		A.~E. Brouwer and H.~Van~Maldeghem.
		\newblock {\em Strongly regular graphs}, volume 182 of {\em Encyclopedia of
			Mathematics and its Applications}.
		\newblock Cambridge University Press, Cambridge, 2022.
		
		\bibitem[CD06]{colbourndinitz}
		C.~J. Colbourn and J.~H. Dinitz, editors.
		\newblock {\em Handbook of Combinatorial Designs}.
		\newblock Taylor Francis, 2006.
		
		\bibitem[CDJHH04]{cockayne2004roman}
		E.~Cockayne, P.~Dreyer~Jr, S.~Hedetniemi, and S.~Hedetniemi.
		\newblock Roman domination in graphs.
		\newblock {\em Discrete mathematics}, 278(1-3):11--22, 2004.
		
		\bibitem[Dem68]{dembowksi68}
		P.~Dembowski.
		\newblock {\em Finite geometries}.
		\newblock Ergebnisse der Mathematik und ihrer Grenzgebiete, Band 44.
		Springer-Verlag, Berlin-New York, 1968.
		
		\bibitem[Den69]{denniston}
		R.~H.~F. Denniston.
		\newblock Some maximal arcs in finite projective planes.
		\newblock {\em J. Combinatorial Theory}, 6:317--319, 1969.
		
		\bibitem[DWSV12]{dewinterschillewaertverstraete}
		S.~De~Winter, J.~Schillewaert, and J.~Verstraete.
		\newblock Large incidence-free sets in geometries.
		\newblock {\em Electron. J. Combin.}, 19(4):Paper 24, 16, 2012.
		
		\bibitem[Ehr15]{ehrenhorg}
		R.~Ehrenhorg.
		\newblock An unbiased marriage theorem.
		\newblock {\em Amer. Math. Monthly}, 122(1):59, 2015.
		
		\bibitem[EPASZ18]{elveyprice}
		A.~Elvey~Price, M.~Adib~Surani, and S.~Zhou.
		\newblock The isoperimetric number of the incidence graph of {${\rm PG}(n,q)$}.
		\newblock {\em Electron. J. Combin.}, 25(3):Paper No. 3.20, 15, 2018.
		
		\bibitem[GRM14]{goldberg}
		F.~Goldberg, D.~Rajendraprasad, and R.~Mathew.
		\newblock Domination in designs.
		\newblock {\em arXiv preprint arXiv:1405.3436}, 2014.
		
		\bibitem[Hae95]{haemers}
		W.~H. Haemers.
		\newblock Interlacing eigenvalues and graphs.
		\newblock {\em Linear Algebra Appl.}, 226/228:593--616, 1995.
		
		\bibitem[HHL20]{hegerhernandezlucas}
		T.~Héger and L.~Hernandez~Lucas.
		\newblock Dominating sets in finite generalized quadrangles.
		\newblock {\em Ars Math. Contemp.}, 19(1):61--76, 2020.
		
		\bibitem[HHS98]{haynes2013fundamentals}
		T.~W. Haynes, S.~T. Hedetniemi, and P.~J. Slater.
		\newblock {\em Fundamentals of domination in graphs}, volume 208 of {\em
			Monographs and Textbooks in Pure and Applied Mathematics}.
		\newblock Marcel Dekker, Inc., New York, 1998.
		
		\bibitem[HK93]{horton1993minimum}
		J.~Horton and K.~Kilakos.
		\newblock Minimum edge dominating sets.
		\newblock {\em SIAM Journal on Discrete Mathematics}, 6(3):375--387, 1993.
		
		\bibitem[HN17]{hegernagy}
		T.~Héger and Z.~L. Nagy.
		\newblock Dominating sets in projective planes.
		\newblock {\em J. Combin. Des.}, 25(7):293--309, 2017.
		
		\bibitem[Hug78]{hughes}
		D.~Hughes.
		\newblock Biplanes and semi-biplanes.
		\newblock In {\em Combinatorial mathematics ({P}roc. {I}nternat. {C}onf.
			{C}ombinatorial {T}heory, {A}ustralian {N}at. {U}niv., {C}anberra, 1977)},
		volume 686 of {\em Lecture Notes in Math.}, pages 55--58. Springer, Berlin,
		1978.
		
		\bibitem[HY13]{henning2013total}
		M.~A. Henning and A.~Yeo.
		\newblock {\em Total domination in graphs}.
		\newblock Springer Monographs in Mathematics. Springer, New York, 2013.
		
		\bibitem[IS06]{ioninshrikhande}
		Y.~J. Ionin and M.~S. Shrikhande.
		\newblock {\em Combinatorics of symmetric designs}, volume~5 of {\em New
			Mathematical Monographs}.
		\newblock Cambridge University Press, Cambridge, 2006.
		
		\bibitem[Kan85]{kantor1985}
		W.~M. Kantor.
		\newblock Classification of {$2$}-transitive symmetric designs.
		\newblock {\em Graphs Combin.}, 1(2):165--166, 1985.
		
		\bibitem[KST54]{kovarisosturan}
		T.~K\"{o}vari, V.~T. S\'{o}s, and P.~Tur\'{a}n.
		\newblock On a problem of {K}. {Z}arankiewicz.
		\newblock {\em Colloq. Math.}, 3:50--57, 1954.
		
		\bibitem[LW99]{laskar}
		R.~Laskar and C.~Wallis.
		\newblock Chessboard graphs, related designs, and domination parameters.
		\newblock {\em J. Statist. Plann. Inference}, 76(1-2):285--294, 1999.
		
		\bibitem[MPS19]{mattheuspavesestorme}
		S.~Mattheus, F.~Pavese, and L.~Storme.
		\newblock On the independence number of graphs related to a polarity.
		\newblock {\em J. Graph Theory}, 92(2):96--110, 2019.
		
		\bibitem[MW07]{mubayiwilliford}
		D.~Mubayi and J.~Williford.
		\newblock On the independence number of the {E}rd{\H{o}}s-{R}\'{e}nyi and
		projective norm graphs and a related hypergraph.
		\newblock {\em J. Graph Theory}, 56(2):113--127, 2007.
		
		\bibitem[MX06]{muzychuckxiang}
		M.~Muzychuk and Q.~Xiang.
		\newblock Symmetric {B}ush-type {H}adamard matrices of order {$4m^4$} exist for
		all odd {$m$}.
		\newblock {\em Proc. Amer. Math. Soc.}, 134(8):2197--2204, 2006.
		
		\bibitem[Pal33]{paley}
		R.~E. A.~C. Paley.
		\newblock On orthogonal matrices.
		\newblock {\em Journal of Mathematics and Physics}, 12(1-4):311--320, 1933.
		
		\bibitem[Sti13]{stinson}
		D.~R. Stinson.
		\newblock Nonincident points and blocks in designs.
		\newblock {\em Discrete Math.}, 313(4):447--452, 2013.
		
		\bibitem[Tho89]{thomason}
		A.~Thomason.
		\newblock Dense expanders and pseudo-random bipartite graphs.
		\newblock In B.~Bollobás, editor, {\em Graph Theory and combinatorics 1988},
		volume~43 of {\em Annals of Discrete Mathematics}, pages 381--386. Elsevier,
		1989.
		
		\bibitem[TZC19]{tang}
		L.~Tang, S.~Zhou, and J.~Chen.
		\newblock Domination number of incidence graphs of block designs.
		\newblock {\em Appl. Math. Comput.}, 363:124600, 6, 2019.
		
		\bibitem[Wil81]{wild81}
		P.~Wild.
		\newblock Divisible semisymmetric designs.
		\newblock In {\em Combinatorial mathematics, {VIII} ({G}eelong, 1980)}, volume
		884 of {\em Lecture Notes in Math.}, pages 346--350. Springer, Berlin-New
		York, 1981.
		
		\bibitem[Wil95]{wild95}
		P.~Wild.
		\newblock Some families of semibiplanes.
		\newblock {\em Discrete Math.}, 138(1-3):397--403, 1995.
		
		\bibitem[YG80]{gavril}
		M.~Yannakakis and F.~Gavril.
		\newblock Edge dominating sets in graphs.
		\newblock {\em SIAM J. Appl. Math.}, 38(3):364--372, 1980.
		
	\end{thebibliography}

\end{document}